\documentclass[12pt]{amsart}

\usepackage{amsfonts}
\usepackage{amsmath}
\usepackage{amssymb}
\usepackage{mathrsfs}
\usepackage{amscd}
\usepackage[all]{xy}
\usepackage[pdftex,final]{graphicx}
\usepackage{hyperref}
\usepackage[hyperpageref]{backref}

\usepackage{exscale,txfonts}

\newtheorem{lem}{Lemma}[section]
\newtheorem{thm}[lem]{Theorem}
\newtheorem{prop}[lem]{Proposition}
\newtheorem{cor}[lem]{Corollary}
\theoremstyle{definition}

\newtheorem{exa}[lem]{Example}
\newtheorem{rem}[lem]{Remark}

\newcommand{\imp}{\Longrightarrow}
\renewcommand{\iff}{\Longleftrightarrow}

\newcommand{\Q}{\Bbb{Q}}
\newcommand{\F}[1]{\Bbb{F}_{#1}}
\newcommand{\bF}[1]{\bar{\Bbb{F}}_{#1}}

\newcommand{\Z}{\Bbb{Z}}

\newcommand{\R}{\Bbb{R}}
\newcommand{\C}{\Bbb{C}}
\newcommand{\K}{\mathcal{K}}

\newcommand{\wn}[1]{\mathcal{W}_{#1}}% set of units u such that 1-u is a unit 
\newcommand{\twn}[1]{\tilde{\mathcal{W}}_{#1}}

\newcommand{\discn}[2]{\Delta(#1,#2)}
\newcommand{\spl}[2]{\mathrm{SL}_{#1}(#2)}

\newcommand{\gl}[2]{\mathrm{GL}_{#1}(#2)}

\newcommand{\pb}[1]{\mathcal{P}(#1)}%  Pre-Bloch Group
\newcommand{\rpb}[1]{{\mathcal{RP}}(#1)}%  refined pre-Bloch group
\newcommand{\rpbker}[1]{\mathcal{RP}_1(#1)}% Ker RP(F) to P(F)
\newcommand{\rrpb}[1]{\overline{\mathcal{RP}}(#1)}% reduced pre-Bloch group (redefined)
\newcommand{\qrpb}[1]{\widetilde{\mathcal{RP}}(#1)}% quasi-reduced refined pre-Bloch group (redefined)
\newcommand{\rrpbker}[1]{\overline{\mathcal{RP}}_1(#1)}
\newcommand{\qrpbker}[1]{\widetilde{\mathcal{RP}}_1(#1)}% reduced Ker lambda_1
\newcommand{\bl}[1]{\mathcal{B}(#1)}% Bloch group
\newcommand{\rbl}[1]{{\mathcal{RB}}(#1)}%refined Bloch group
\newcommand{\rrbl}[1]{\overline{\mathcal{RB}}(#1)}%reduced refined Bloch group (redefined)
\newcommand{\qrbl}[1]{\widetilde{\mathcal{RB}}(#1)}%quasi-reduced refined Bloch group (redefined)
\newcommand{\ppb}[1]{\mathcal{A}(#1)} %pre-pre-Bloch group
\newcommand{\gpb}[1]{\left[ #1\right]} %notation for generators of (pre)Bloch
\newcommand{\suss}[2]{\psi_{#1}\left( #2\right)}

\newcommand{\ks}[2]{{\K}^{\mbox{\tiny $(#1)$}}_{#2}}

\newcommand{\kv}[1]{\mathcal{L}_{#1}}
\newcommand{\pbconst}[1]{\tilde{C}_{#1}}
\newcommand{\pcconst}[1]{\tilde{D}_{#1}}
\newcommand{\bconst}[1]{C_{#1}}
\newcommand{\cconst}[1]{D_{#1}}
\newcommand{\cconstmod}[1]{\mathcal{D}_{#1}}

\newcommand{\rcr}{\mathrm{cr}}
\newcommand{\bw}{\Lambda}
\newcommand{\rbw}{\widetilde{\bw}}

\newcommand{\spec}[1]{S_{#1}}
\newcommand{\pn}[2]{\mathbb{P}^{#1}(#2)}
\newcommand{\projl}[1]{\pn{1}{#1}}

\newcommand{\indf}[3]{\ind{#2}{#1}{#3}}

\newcommand{\sq}[1]{\unitr{#1}/(\unitr{#1})^2}% old 
\newcommand{\tnorm}[1]{\mathcal{N}_{#1}}% +- norms from F(sqrt(-3))
\newcommand{\rsq}[1]{{\mathcal{G}}_{#1}}% F^\times mod +- norms from F(sqrt(-3))

\newcommand{\chara}[1]{\mathrm{char}(#1)}

\newcommand{\<}{\langle}
\renewcommand{\>}{\rangle}
\newcommand{\il}[1]{\< #1 \>}
\newcommand{\card}[1]{\left| #1 \right|}
\newcommand{\pset}[2]{\{#1 : #2 \}}
\newcommand{\sset}[1]{\{ #1 \}}
\renewcommand{\forall}{\mbox{ for all }}

\newcommand{\dual}[1]{\widehat{#1}}

\newcommand{\id}[1]{\mathrm{Id}_{#1}}

\renewcommand{\ker}[1]{\mathrm{Ker}(#1)}
\newcommand{\image}[1]{\mathrm{Im}(#1)}

\newcommand{\coker}[1]{\mathrm{Coker}(#1)}

\newcommand{\unitr}[1]{{#1}^{\times}}

\newcommand{\modgen}[3]{\< #1\ |\ #2\>_{#3}}%module generated by a set

\newcommand{\ptor}[2]{{#1}[#2]}

\newcommand{\sgr}[1]{\mathrm{R}_{#1}}% group ring of square classes of field
\newcommand{\rsgr}[1]{\widehat{\sgr{#1}}} % `reduced' group ring 
\newcommand{\an}[1]{\left\langle{#1}\right\rangle}
\newcommand{\pf}[1]{\left\langle\!\left\langle{#1}\right\rangle\!\right\rangle}
\newcommand{\aug}[1]{\mathcal{I}_{#1}}
\newcommand{\pows}[2]{#1\left[\!\left[ #2 \right]\!\right]}%Power series ring
\newcommand{\laurs}[2]{{#1}\left(\!\left( #2 \right)\!\right)}%Laurent series field

\newcommand{\igr}[1]{\Z[#1]}

\newcommand{\pp}[1]{\mathrm{p}_{#1}^+}
\newcommand{\ppm}[1]{\mathrm{p}_{#1}^-}
\newcommand{\ep}[1]{\mathrm{e}_{#1}^+} % +-idempotent
\newcommand{\epm}[1]{\mathrm{e}_{#1}^-} % - idempotent

\newcommand{\dwn}[2]{{#1}_{(#2)}}

\newcommand{\matr}[4]{\left[\begin{array}{cc}
#1&#2\\
#3&#4\\
\end{array}
\right]}

\newcommand{\col}[2]{\left[\begin{array}{c}
#1\\
#2\\
\end{array}
\right]}

\newcommand{\vect}[1]{\mathbf{#1}}

\newcommand{\Extpow}[3]{\wedge^{#1}_{#2}(#3)}

\newcommand{\sym}[3]{\mathrm{Sym}^{#1}_{#2}(#3)}

\newcommand{\asymm}{\circ}
\newcommand{\asym}[3]{\mathrm{S}^{#1}_{#2}(#3)}
\newcommand{\qasym}[3]{\tilde{\mathrm{S}}^{#1}_{#2}(#3)}

\newcommand{\zhalf}[1]{{#1}\left[ \tfrac{1}{2}\right]}
\newcommand{\zn}[2]{{#1}\left[ \tfrac{1}{#2}\right]}
\newcommand{\zzhalf}[1]{{#1}_{\mathrm{ odd}}}

\newcommand{\genu}[2]{\mathcal{U}^{\mathrm{gen}}_{#1}(#2)}
\newcommand{\gnu}[1]{\mathcal{U}^{\mathrm{gen}}_{#1}}

\newcommand{\res}[1]{\arrowvert_{#1}}

\newcommand{\milk}[2]{K^{\mathrm{\small M}}_{#1}({#2})}
\newcommand{\kind}[1]{K^{\mathrm{\small ind}}_3(#1)}
\newcommand{\ho}[3]{\mathrm{H}_{#1}(#2,#3 )}
\newcommand{\hoz}[2]{\ho{#1}{#2}{\Z}}
\newcommand{\hot}[2]{\mathrm{H}_3\left( \spl{2}{#1},#2\right)_0}

\newcommand{\ind}[2]{\mathrm{Ind}^{#1}_{#2}}

\newcommand{\Ind}[3]{\mathrm{Ind}_{#1}^{#2}#3}

\newcommand{\Tor}[2]{\mathrm{Tor}^{\Z}_{1}(#1,#2)}

\title{The third homology of $\mathrm{SL}_2$ of local rings}
\author{Kevin Hutchinson}
\address{School of Mathematics and Statistics,
University College Dublin,
Ireland}
\email{kevin.hutchinson@ucd.ie}
\date{\today}


\begin{document}
\bibliographystyle{plain}
\maketitle

\begin{abstract}
We describe the third homology of $SL_2$ of local rings, over $\Z\left[\tfrac{1}{2}\right]$, in terms of a 
refined Bloch group. We use this to derive a localization sequence for the third homology of $SL_2$ of 
certain discrete valuation rings.
\end{abstract}
%\tableofcontents

\section{Introduction}\label{sec:intro}
\subsection{Overview and background} 
This paper is concerned with calculating and understanding the structure of 
the third homology of $\mathrm{SL}_2$, with coefficients $\zhalf{\Z}$ of local integral domains.
We generalise the main theorem of \cite{hut:cplx13} from fields 
to local domains. We then combine this with the main result of \cite{hut:arxivh3sl2dv} to 
prove a localization exact sequence for the third homology of $\mathrm{SL}_2$ of 
 certain discrete valuation rings.

In their study of Hilbert's third problem for scissors congruence of $3$-dimensional hyperbolic 
polyhedra, Dupont and Sah \cite[Appendix]{sah:dupont} proved the \emph{Bloch-Wigner} theorem: 
There is a short exact sequence
\[
0\to \mu_\C\to \ho{3}{\spl{2}{\C}}{\Z}\to \bl{\C}\to 0
\] 
where $\bl{F}$ denotes the Bloch group of the field $F$. By definition, $\bl{F}$ is the kernel 
of the Dehn invariant $\lambda: \pb{F}\to (F^\times\otimes F^\times)/\il{x\otimes y+y\otimes x}:=
\asym{2}{\Z}{F^\times}$ 
where $\pb{F}$ is the \emph{scissors congruence group} of the field $F$ and is given by an explicit 
presentation. For details, see section \ref{sec:bloch} below. 

In fact the natural map from $\ho{3}{\spl{2}{\C}}{\Z}$ to $\kind{\C}$, the indecomposable 
$K_3$ of $\C$, is an isomorphism. 
From this viewpoint, 
Suslin (\cite[Theorem]{sus:bloch}) subsequently generalised the Bloch-Wigner theorem
 to all (infinite) fields $F$: There is a natural exact sequence 
\[
0\to \widetilde{\mathrm{tor}(\mu_F,\mu_F)}\to \kind{F}\to \bl{F}\to 0
\]
where $\widetilde{\mathrm{tor}(\mu_F,\mu_F)}$ is the unique nontrivial extension of 
$\mathrm{tor}(\mu_F,\mu_F)$ by $\Z/2$ if the characteristic 
of $F$ is not $2$, and 
$\widetilde{\mathrm{tor}(\mu_F,\mu_F)}=\mathrm{tor}(\mu_F,\mu_F)$ in characteristic $2$.
Suslin's result has since been generalized to \emph{rings with many units} by Mirzaii and Mokari 
(\cite{mirzaii:mok2}). This class of rings includes, in particular, local rings with infinite residue
field. Mirzaii has also recently extended  the theorem to local rings with (sufficiently 
large) finite residue field (\cite{mirzaii:arxivbw}).   

However, for general fields, the natural map $\ho{3}{\spl{2}{F}}{\Z}\to\kind{F}$ is surjective 
but has a large kernel (see \cite{hut:rbl11}, for example). The integral homology of $\spl{2}{F}$ 
is naturally a module for the group ring $\sgr{F}:=\Z[\sq{F}]$ of square classes. This action is 
generally nontrivial, while the corresponding natural action of $\sgr{F}$ on $\kind{F}$ is the trivial 
action. The Bloch-Wigner sequence of Dupont-Sah, viewed as a statement about $\ho{3}{\spl{2}{F}}{\Z}$
was generalised to arbitrary fields in \cite{hut:cplx13}. For any field $F$ there is a natural 
short exact sequence of $\sgr{F}$-modules 
\begin{eqnarray}\label{bwA}
0\to \zhalf{\mathrm{tor}(\mu_F,\mu_F)}\to \ho{3}{\spl{2}{F}}{\zhalf{\Z}}\to \zhalf{\rbl{F}}\to 0.
\end{eqnarray}
(To avoid $2$-torsion complications, we state the result for $\zhalf{\Z}$-coefficients.) Here 
$\rbl{F}$ is the \emph{refined Bloch group} of $F$. It is an $\sgr{F}$-submodule of $\rpb{F}$, 
an $\sgr{F}$-module given by an explicit
presentation analagous to that defining the $\Z$-module $\pb{F}$ (see section \ref{sec:bloch}). 
In particular, applying $F^\times$-coinvariants to this exact sequence we recover Suslin's Bloch-Wigner 
exact sequence for $\kind{F}$, at least over $\zhalf{\Z}$. 

The first main result below (Theorem \ref{thm:h3sl2A}) is a generalization of   exact sequence
(\ref{bwA}) from fields to local integral domains. 
The residue field may be finite, apart from a few small 
exceptions arising from some limitations of our method of proof. As an application, and using the 
results of Mirzaii and Mokari referred to above, we prove (Corollary \ref{cor:bwfin}) 
that for any local domain $A$ with sufficiently large residue field the natural 
map $\ho{3}{\spl{2}{A}}{\Z}\to \kind{A}$ induces an isomorphism 
\[
\ho{3}{\spl{2}{A}}{\zhalf{\Z}}_{A^\times}\cong \zhalf{\kind{A}}.
\]

The second half of the paper is devoted to proving a localization theorem (Theorem \ref{thm:loc}) 
relating the third homology of $\mathrm{SL}_2$ of certain discrete valuation rings to the 
third homology of $\mathrm{SL}_2$ of the field of quotients: Let $F$ be a field with discrete 
valuation $v$. Let $\mathcal{O}_v$ be the associated valuation ring and let $k$ be the residue field.
Let $U_1$ be the group of units in $\mathcal{O}_v$ which map to $1$ in $k$. Suppose that 
$U_1=U_1^2$. We assume also another technical condition, affecting only $3$-torsion. 
Then there is a natural exact sequence
\begin{eqnarray}\label{loc}
0\to \ho{3}{\spl{2}{\mathcal{O}_v}}{\zhalf{\Z}}
\to \ho{3}{\spl{2}{F}}{\zhalf{\Z}}\to 
\zhalf{\rpbker{k}}\to 0.
\end{eqnarray}
Here the module $\rpbker{k}$ is the \emph{refined scissors congruence group}. It is a 
 submodule of $\rpb{k}$ which contains $\rbl{k}$. It is 
defined in section \ref{sec:bloch} below. It has 
already arisen elsewhere in connection with the calculation and properties of the third homology of 
$\mathrm{SL}_2$: in \cite{hut:laurent} it is shown that for any infinite field $k$ there is a 
natural short exact 
\[
0\to \ho{3}{\spl{2}{k}}{\zhalf{\Z}}\to \ho{3}{\spl{2}{k[t,t^{-1}]}}{\zhalf{\Z}}\to 
\zhalf{\rpbker{k}}\to 0.
\]  
 
The module $\rpbker{k}$ can be calculated explicitly in terms of classical scissors congruence 
groups (see \cite{hut:arxivh3sl2dv} for many examples). Using such calculations, we obtain 
calculations of\\
 $\ho{3}{\spl{2}{A}}{\zhalf{\Z}}$ for certain local domains. For example 
\[
\ho{3}{\spl{2}{\Z_p}}{\zhalf{\Z}}\cong \zhalf{\kind{\Q_p}}
\]
for any prime number $p\geq 11$, while for any prime $p$ 
\[
\ho{3}{\spl{2}{\pows{\Q_p}{x}}}{\zhalf{\Z}}\cong \zhalf{\kind{\laurs{\Q_p}{x}}}\oplus \zhalf{\pb{\F{p}}}
\]
 where the second factor on the right is a cyclic group whose  order is the odd part of $p+1$.

\begin{rem}
The condition $U_1=U_1^2$ in Theorem \ref{thm:loc} 
arises from our appeal to the results of \cite{hut:arxivh3sl2dv}. 
We would expect that this sequence remains exact for a much larger class of local domains.   
\end{rem}

\begin{rem} We refer to (\ref{loc}) as a `localization sequence' because of the analogy with 
$K$-theory: For a discrete valuation ring $A$ with field of quotients $F$ there is a natural 
isomorphism for any $m\geq 3$
\[
\ho{3}{\spl{m}{A}}{\zhalf{\Z}}\cong\zhalf{K_3(A)}
\]
(by the stabilization results in \cite{schlicht:arxiveuler}, for example). It follows that 
for any $m\geq 3$ there is a short exact sequence 
\[
0\to \ho{3}{\spl{m}{A}}{\zhalf{\Z}}\to \ho{3}{\spl{m}{F}}{\zhalf{\Z}} \to 
\zhalf{K_2(k)}\to 0
\]
which comes from the $K$-theory localization sequence. Note that when $m=2$, we must replace 
$K_2(k)$ with the unrelated functor $\rpbker{k}$. ( In \cite{hut:laurent} it is shown that the 
natural map $\rpbker{k}\to K_2(k)$, arising from homology stabilization, is the zero map.) 
This latter functor however seems to bear
 the same relation 
to the third homology of $\mathrm{SL}_2$ as $K_2$ does to the third homology of $\mathrm{SL}_m$ 
for large $m$. 
\end{rem}
\subsection{Layout of article}

The article is laid out as follows:

In section \ref{sec:bloch} we review the definitions of (refined) scissors congruence groups and 
Bloch groups of commutative rings and recall the required results 
from \cite{hut:cplx13} and \cite{hut:rbl11}. 

In section \ref{sec:h3sl2A} we prove (Theorem \ref{thm:h3sl2A}) 
that for local domain $A$ with sufficiently large residue field there is a natural short exact 
sequence
\[
0\to \mathrm{tor}\zhalf{(\mu_A,\mu_A)}\to \ho{3}{\spl{2}{A}}{\zhalf{\Z}}\to \zhalf{\rbl{A}}\to 0.
\]
The proof follows the same route as the proof of the corresponding theorem for fields 
in \cite{hut:cplx13}, but we 
supply all the necessary details for the convenience of the reader. In subsection \ref{sec:bwfin}, 
we combine this theorem with the recent results of Mirzaii to show that  $\zhalf{\kind{A}}$ is 
obtained from $\ho{3}{\spl{2}{A}}{\zhalf{\Z}}$ by taking coinvariants for the action of 
$A^\times$. This was not known previously in the case of  local domains with finite residue field. 

In section \ref{sec:rpbA}, we consider certain  submodules, quotient modules and special 
elements of the refined 
scissors congruence 
group $\rpb{A}$ which play an important role in our calculations. Again, we are here following 
a route already 
covered in the case of fields (in \cite{hut:rbl11}). 
Only small adaptations are needed to extend the results for fields to the more general
case of local rings, but we include (most) details for the reader's convenience. 

In the short section 
\ref{sec:char}, we review certain character-theoretic techniques from \cite{hut:arxivh3sl2dv} 
which we will need in Section \ref{sec:val}. 

In section \ref{sec:val} we use the results of the two preceding sections to 
show that for a local domain $A$ with (sufficiently large) residue field $k$ the reduction 
homomorphism $A\to k$ induces an isomorphism $\hot{A}{\zhalf{\Z}}\cong \hot{k}{\zhalf{\Z}}$, 
up to some possible $3$-torsion, and under the condition $U_1=U_1^2$.  Here, 
for any ring $A$,  
$\hot{A}{\Z}$ denotes the kernel of the map $\ho{3}{\spl{2}{A}}{\Z}\to\kind{A}$. 

In section \ref{sec:local} we use the results of the previous section to prove the localization theorem 
(Theorem\ref{thm:loc}) discussed above. We finish be applying these  results to the explicit 
calculation of $\ho{3}{\spl{2}{A}}{\zhalf{\Z}}$ for discrete valuation rings $A$. 

In an  appendix we verify that the proof of  \emph{key 
identity} $\pf{x}\cconstmod{A}=\suss{1}{x}-\suss{2}{x}$ in $\rpb{A}$ 
(see Theorem \ref{thm:df}) which was given for fields in \cite{hut:rbl11}, extends to general local 
domains. 
This identity is  crucial to  the  calculations of the second half of the article.

\subsection{Some Notation}

For a commutative ring $A$, we let $\unitr{A}$ denote the group of units of $A$.
For $x\in \unitr{A}$, we will let $\an{x} \in \sq{A}$ denote the corresponding square class. 
Let $\sgr{A}$ denote the integral group ring $\igr{\sq{A}}$ of the group $\sq{A}$. 
We will use the notation $\pf{x}$ for the basis elements, $\an{x}-1$, 
of the augmentation ideal $\aug{A}$ of $\sgr{A}$.

For any $a\in \unitr{A}$, we will let $\pp{a}$ and $\ppm{a}$ denote respectively 
the elements $1+\an{a}$ and $1-\an{a}$ 
in $\sgr{A}$. 

For any abelian group $G$ we will let $\zhalf{G}$ denote $G\otimes\zhalf{\Z}$. 

For an integer $n$, we will let 
$\zzhalf{n}$ denote the odd part of $n$. 
Thus if $G$ is a finite abelian group of order $n$, then $\zhalf{G}$ is 
a finite abelian group of order $\zzhalf{n}$.

We let $\ep{a}$ and $\epm{a}$ denote 
respectively the mutually orthogonal idempotents
\[
\ep{a}:=\frac{\pp{a}}{2}=\frac{1+\an{a}}{2},\quad 
\epm{a}:=\frac{\ppm{a}}{2}=\frac{1-\an{a}}{2}\in\zhalf{\sgr{A}}. 
\] 
(Of course, these operators depend only on the class of $a$ in $\sq{A}$.)

For a prime power $q$, the finite field with $q$ elements is denoted $\F{q}$.

Given an abelian group $G$ we let  $\asym{2}{\Z}{G}$ denote the group 
\[
\frac{G\otimes_{\Z}G}{<x\otimes y + y\otimes x | x,y \in G>}
\]
and, for $x,y\in G$, we denote by $x\asymm y$ the image of $x\otimes y$ in $\asym{2}{\Z}{G}$.

\section{Scissors congruence groups and Bloch Groups}\label{sec:bloch}
In this section we review some known results about the third homology of $\mathrm{SL}_2$ and its 
relation to $K$-theory and scissors congruence groups.

\subsection{Indecomposable $K_3$}
Let $A$ be a either a local ring or a field. Let $K_\bullet(A)$ denote the Quillen $K$-theory of $A$ and let 
$\milk{\bullet}{A}$ be the Milnor $K$-theory. There is a natural homomorphism of graded rings 
$\milk{\bullet}{A}\to K_\bullet(A)$. Indecomposable $K_3$ of $A$ is the group 
$\kind{A}:=\coker{\milk{3}{A}\to K_3(A)}$. We will require the following theorem from $K$-theory:
  
\begin{thm}\label{thm:ktheory}
Let $A$ be discrete valuation ring  with field of fractions $K$ and residue field $F$. Suppose that 
either $\mathrm{char}(K)=\mathrm{char}(F)$ or that $k$ is algebraic over $\F{p}$. Then the inclusion
$A\to K$ induces an isomorphism $\kind{A}\cong\kind{K}$.  
\end{thm}
\begin{proof} Let $F$ be the residue field of $A$ and let $\pi$ be a uniformizer.
There is a commutative diagram with exact columns
\[
\xymatrix{
&
%0\ar[d]
&
0\ar[d]
&&\\
&
\milk{3}{A}\ar[r]\ar[d]
&
\milk{3}{K}\ar^-{\delta_\pi}[r]\ar[d]
&
\milk{2}{F}\ar[r]\ar^-{=}[d]
&
0\\
0\ar[r]
&
K_3(A)\ar[r]\ar[d]
&
K_3(K)\ar^-{\delta_\pi}[r]\ar[d]
&
K_2(F)\ar[r]
&
0\\
&
\kind{A}\ar[r]\ar[d]
&
\kind{K}\ar[d]
&&\\
&
0
&
0
&&\\
}
\]
The exactness of the top row is well-known (see \cite[Chapter V,Corollary 6.6.2]{weibel:kbook} for example).
The second row is exact by Gersten's conjecture,
 which is known for the case of equicharacteristic discrete valuation rings or discrete valuation 
rings with finite residue fields (see, for 
example, \cite[Chapter V, p. 454]{weibel:kbook}). The result follows by the snake lemma.
\end{proof}

\subsection{Scissors congruence groups and  Bloch groups}

For a commutative ring  $A$, let $\wn{A}$ denote the set $\pset{u\in \unitr{A}}{1-u\in \unitr{A}}$.
The \emph{scissors congruence group} or \emph{pre-Bloch group}, $\pb{A}$, is the group 
generated 
by the elements $\gpb{x}$, $x\in \wn{A}$,  subject to the relations 
\[
R_{x,y}:\quad  \gpb{x}-\gpb{y}+\gpb{\frac{y}{x}}-\gpb{\frac{1-x^{-1}}{1-y^{-1}}}+\gpb{\frac{1-x}{1-y}} \quad 
\forall x,y, x/y \in \wn{A}.
\]

For a commutative ring $A$, the map 
\[
\lambda:\pb{A}\to \asym{2}{\Z}{\unitr{A}},\quad  [x]\mapsto \left(1-{x}\right)\asymm {x}
\]
is well-defined, and the \emph{Bloch group of $A$}, $\bl{A}\subset \pb{A}$, is defined to be the kernel of $\lambda$. 

For a field $F$, the Bloch group is known to be closely related to the indecomposable $K_3$: 
There is a natural exact sequence  
\[
0\to \widetilde{\Tor{\mu_F}{\mu_F}}\to \kind{F}\to \bl{F} \to 0
\]
where $ \widetilde{\Tor{\mu_F}{\mu_F}}$ is the unique nontrivial extension of $\Tor{\mu_F}{\mu_F}$ by $\Z/2$. (See 
Suslin \cite{sus:bloch} for infinite fields and \cite{hut:cplx13} for finite fields.)

More recently, this has been generalized to a much larger class of rings by B. Mirzaii and 
F. Y. Mokari:

\begin{thm}[Mirzaii-Mokari, \cite{mirzaii:mok2}, Mirzaii \cite{mirzaii:arxivbw}]
\label{thm:mirzaii}
Let $A$ be an integral domain with many units or a local ring whose residue field 
has order at least $11$ and is not $16$ or $32$ then there is a natural exact sequence
\[
0\to \widetilde{\Tor{\mu_A}{\mu_A}}\to \kind{A}\to \bl{A} \to 0.
\]
\end{thm}
\begin{rem} We will not define the term \emph{ring with many units}; for the present purposes 
it is enough to know that a local ring with infinite residue field is a ring with many units.
\end{rem}

\subsection{The refined Bloch group and $\ho{3}{\spl{2}{A}}{\Z}$}\label{subsec:rbl}

For any  ring $A$ there is a natural  homomorphism
\begin{eqnarray}\label{eqn:kind}
\hoz{3}{\spl{2}{A}}\to \kind{A}
\end{eqnarray}
which is surjective if $A$ is a field or commutative local ring.
We let $\hot{A}{\Z}$ denote the kernel of this map.

For any commutative ring $A$, the group extension 
\[
1\to \spl{2}{A}\to\gl{2}{A}\to \unitr{A}\to 1
\]
induces an action -- by conjugation -- of $\unitr{A}$ on $\hoz{\bullet}{\spl{2}{A}}$ which 
factors through $\sq{A}$. 

The following is \cite[Theorem 3.7]{mirzaii:thirdmany}:
\begin{prop}\label{prop:mirzaii}
For a ring with many units $A$ the map (\ref{eqn:kind}) induces an isomorphism
\[
\ho{3}{\spl{2}{A}}{\zhalf{\Z}}_{\unitr{A}}\cong\zhalf{\kind{A}}.
\]
\end{prop}

\begin{rem} We extend this result to local domains with finite residue field below (Corollary 
\ref{cor:bwfin}).
\end{rem}

\begin{cor}\label{cor:mirzaii}
Let $A$ be a local ring with infinite residue field. Let 
$\hot{A}{\Z}$ denote the kernel of the surjective homomorphism 
$\hoz{3}{\spl{2}{A}}\to \kind{A}$. Then 
\[
\hot{A}{\zhalf{\Z}}=\aug{A}\ho{3}{\spl{2}{A}}{\zhalf{\Z}}. 
\]
\end{cor}

However, as our calculations -- in \cite{hut:rbl11} and below -- 
show, the action of $\sq{A}$ on 
$\hoz{3}{\spl{2}{A}}$ is in general non-trivial.

Thus
$\hoz{3}{\spl{2}{A}}$ is naturally an $\sgr{A}$-module, and for general fields or rings, 
in order to give a Bloch-type description
of it, we must incorporate the $\sgr{A}$-module structure at each stage of the process.

Let $A$ be a commutative ring. We let \
\[
\wn{A}:= \{ u\in A^\times\ |\ 1-u\in A^\times \}.
\]
Then we define $\rpb{A}$
to be  the $\sgr{A}$-module with generators $\gpb{x}$, $x\in \wn{A}$ 
subject to the relations  
\[
S_{x,y}:\quad 0=\gpb{x}-\gpb{y}+\an{x}\gpb{ \frac{y}{x}}-\an{x^{-1}-1}\gpb{\frac{1-x^{-1}}{1-y^{-1}}}
+\an{1-x}\gpb{\frac{1-x}{1-y}},\quad \forall x,y, y/x\in \wn{A}
\]

Of course, from the definition it follows immediately that $\pb{A}=(\rpb{A})_{\unitr{A}}=
\ho{0}{\unitr{A}}{\rpb{A}}$.

Let  $\bw= (\lambda_1,\lambda_2)$ 
be the $\sgr{A}$-module homomorphism 
\[
\rpb{A}\to \aug{A}^2\oplus\asym{2}{\Z}{\unitr{A}}
\]
where 
$\lambda_1:\rpb{A}\to \aug{A}^2$ is the 
map $\gpb{x}\mapsto \pf{1-x}\pf{x}$, and $\lambda_2$ is the composite
\[
\xymatrix{
\rpb{A}\ar@{>>}[r]
&\pb{A}\ar[r]^-{\lambda}
&\asym{2}{\Z}{\unitr{A}}.
}
\] 
(It is straightforward to  verify directly  that $\bw$ is well-defined module homomorphism.) 

The  
\emph{refined Bloch group} of the commutative ring $A$ (with at least $4$ elements) 
is the $\sgr{A}$-module 
\[
\rbl{A}:=\ker{\bw: \rpb{A}\to \aug{A}^2\oplus\asym{2}{\Z}{\unitr{A}}}.
\]

For a field $F$, the refined Bloch group bears the same relation to 
$\ho{3}{\spl{2}{F}}{\zhalf{\Z}}$ as the classical Bloch group does to $\kind{F}$:
\begin{thm}[\cite{hut:cplx13}]\label{thm:BW}
Let $F$ be a field with  at least $29$ elements. Then there is a natural short exact 
sequence
\[
0\to \zhalf{\Tor{\mu_F}{\mu_F}}\to \ho{3}{\spl{2}{F}}{\zhalf{\Z}}\to \zhalf{\rbl{F}} \to 0.
\]
\end{thm}

\begin{rem}
If we take coinvariants for the action of $F^\times$ in this sequence, the resulting sequence 
is exact, and it 
is precisely the sequence of Theorem \ref{thm:mirzaii} tensored with $\zhalf{\Z}$. 
\end{rem}

%\subsection{The refined Bloch Group and $\ho{3}{\spl{2}{F}}{\Z}$}

\subsection{The refined scissors congruence group $\rpbker{A}$}

The \emph{refined scissors congruence group of $A$} is defined to be the $\sgr{A}$ module 
\[
\rpbker{A}:= \ker{\lambda_1:\rpb{A}\to \aug{A}^2}. 
\]

We will require the following facts below:

\begin{prop}\label{prop:rblrpbker} Let $A$ be a commutative ring. Then
\begin{enumerate}
\item The natural map $\rpbker{A}\to \pb{A}$ induces an isomorphism
\[
\zhalf{\rpbker{A}}_{\unitr{A}}\cong \zhalf{\pb{A}}.
\]
\item The natural map $\rbl{A}\to \bl{A}$ induces an isomorphism
\[
\zhalf{\rbl{A}}_{\unitr{A}}\cong \zhalf{\bl{A}}.
\]
\item The inclusion map $\rbl{A}\to\rpbker{A}$ induces an isomorphism
\[
\aug{A}\zhalf{\rbl{A}}\cong \aug{A}\zhalf{\rpbker{A}}
\]
\end{enumerate}
\end{prop}

\begin{proof} Since $\sq{A}$ is annihilated by $2$ it follows that taking 
$\sq{A}$-coinvariants is an exact functor on the category of $\zhalf{\sgr{A}}$-modules.
\begin{enumerate}
\item Taking $\sq{A}$-coinvariants of the exact sequence 
$0\to \zhalf{\rpbker{A}}\to \zhalf{\rpb{A}}$ gives an injective map 
\[
0\to \zhalf{\rpbker{A}}_{\unitr{A}}\to \zhalf{\rpb{A}}_{\unitr{A}}=\zhalf{\pb{A}}.
\]

On the other hand, the map $\zhalf{\rpbker{A}}\to\zhalf{\pb{A}}$ is surjective: Let 
$a\in\wn{A}$. Direct calculation shows that 
\[
g(a):=\pp{-1}\gpb{a}+\pf{1-a}(\gpb{a}+\an{-1}\gpb{a^{-1}})\in\ker{\lambda_1}=\rpbker{A}
\]
and $g(a)$ maps to $2\gpb{a}\in \pb{A}$. 
\item Taking $\sq{A}$-coinvariants of the exact sequence of $\zhalf{\sgr{A}}$-modules
\[
\xymatrix{
0\ar[r]
& \zhalf{\rbl{A}}\ar[r]
& \zhalf{\rpbker{A}}\ar[r]^-{\lambda_2} 
&\zhalf{\asym{2}{\Z}{\unitr{A}}}\\
}
\] 
gives the exact sequence 
\[
\xymatrix{
0\ar[r]
& \zhalf{\rbl{A}}_{\unitr{A}}\ar[r]
& \zhalf{\rpbker{A}}_{\unitr{A}}=\zhalf{\pb{A}}\ar[r]^-{\lambda} 
&\zhalf{\asym{2}{\Z}{\unitr{A}}}\\
}
\]
and hence 
\[
\zhalf{\rbl{A}}_{\unitr{A}}=\ker{\lambda:\zhalf{\pb{A}}\to \zhalf{\asym{2}{\Z}{\unitr{A}}}}=
\zhalf{\bl{A}}.
\]
\item Let $\Sigma(A):= \image{\lambda:\pb{A}\to \asym{2}{\Z}{\unitr{A}}}$. Then by 
(1) and (2) there is a commutative diagram of $\zhalf{\sgr{A}}$-modules with exact rows and 
columns:
\[
\xymatrix{
&
0\ar[d]
&
0\ar[d]
&&\\
&\aug{A}\zhalf{\rbl{A}}\ar[r]\ar[d]
&\aug{A}\zhalf{\rpbker{A}}\ar[d]
&&\\
0\ar[r]
&\zhalf{\rbl{A}}\ar[r]\ar[d]
&\zhalf{\rpbker{A}}\ar[r]^-{\lambda_2}\ar[d]
&\zhalf{\Sigma(A)}\ar[r]\ar[d]^-{=}
&0\\
0\ar[r]
&\zhalf{\bl{A}}\ar[r]\ar[d]
&\zhalf{\pb{A}}\ar[r]^-{\lambda}\ar[d]
&\zhalf{\Sigma(A)}\ar[r]
&0\\
&0
&0
&&\\
}
\]
\end{enumerate} 
\end{proof}
\section{The third homology of $\mathrm{SL}_2$ of local rings}\label{sec:h3sl2A}

In this section 
we generalize Theorem \ref{thm:BW} to local domains with sufficiently large residue field 
(Theorem \ref{thm:h3sl2A} below). The 
proof follows closely the proof of the corresponding for fields given in \cite{hut:rbl11}. 
The first step is to give a homological description 
of the the modules $\rpb{A}$ and $\rpbker{A}$.

\subsection{Alternative description  of $\rpb{A}$} 
Let $A$ be a commutative ring. A row vector $\vect{u}=(u_1,u_2)\in A^2$ is said to be \emph{unimodular} if 
$Au_1+Au_1=A$. Equivalently, $\vect{u}$ is unimodular if there exists $\vect{v}\in A^2$ such that 
\[
\col{\vect{u}}{\vect{v}}\in\gl{2}{A}.
\] 

We let $U_2=U_2(A)$ denote the set of $2$-dimensional unimodular row vectors of $A$.  $U_2$ is a 
right $\gl{2}{A}$-set. 
In particular this induces an action of $\unitr{A}=Z(\gl{2}{A})$ acting as multiplication by scalars.

Let 
\[
\gnu{n}=\genu{n}{A}:= \left\{ (\vect{u}_1,\ldots,\vect{u}_n)\in U_2^n : \col{\vect{u}_i}{\vect{u}_j}\in\gl{2}{A} 
\mbox{ for all }i\not= j\right\}.
\]

$(\unitr{A})^n$ acts entry-wise on $\gnu{n}$ and we  let $X_n=X_n(A)=\gnu{n}/(\unitr{A})^n$. 
Observe that $\gnu{n}$ and 
$X_n$ are right $\gl{2}{A}$-sets (with the natural diagonal action).  

In particular, $X_1=U_2/\unitr{A}$ and $X_n\subset X_1^n$. If $\vect{u}=(u_1,u_2)\in U_2$ we will denote the 
corresponding class in $X_1$ by $\bar{\vect{u}}$ or $[u_1,u_2]$.

We have two natural injective maps from $A$ to $X_1$:
\[
\iota_+:A\to X_1,\ a\mapsto a_+:=[a,1]\mbox{ and }\iota_-:A\to X_1,\ a\mapsto a_-:=[1,a]
\]
Clearly, $a_+=b_-$ in $X_1$ if and only if $a,b\in \unitr{A}$ and $b=a^{-1}$. We will identify $\unitr{A}$ with 
its image in $X_1$ under the map $\iota_+$.  

If $(\vect{u},\vect{v})\in \gnu{2}$, we set
\[
d(\vect{u},\vect{v}):=\mathrm{det}\left( \col{\vect{u}}{\vect{v}}\right) \in \unitr{A}
\]
and
\[
T_{\vect{u},\vect{v}}:= \col{\vect{u}}{\vect{v}}^{-1}\cdot\matr{0}{-1}{d(\vect{u},\vect{v})}{0}\in\spl{2}{A}.
\]

Thus, for $(\vect{u},\vect{v})\in \gnu{2}$ with $d=d(\vect{u},\vect{v})$ we have 
\[
\vect{u}\cdot T_{\vect{u},\vect{v}}= (0,-1)\mbox{ and } \vect{v}\cdot T_{\vect{u},\vect{v}}= (d,0)\mbox{ in } U_2
\]
and hence 
\[
\bar{\vect{u}}\cdot T_{\vect{u},\vect{v}}= 0_+,\ \bar{\vect{v}}\cdot T_{\vect{u},\vect{v}}= 0_-\mbox{ in } X_1.
\]

Let $\phi: \gnu{3}\to X_1$ be the map defined by 
\[
\phi(\vect{u},\vect{v},\vect{w}):= \bar{\vect{w}}\cdot T_{\vect{u},\vect{v}}\in X_1.
\]

\begin{lem}\label{lem:phi}
Let $A$ be a commutative ring and let $(\vect{u},\vect{v},\vect{w})\in \gnu{3}$. Then 
\[
\phi(\vect{u},\vect{v},\vect{w})=
\left( \frac{d(\vect{u},\vect{w})\cdot d(\vect{u},\vect{v})}{d(\vect{v},\vect{w})}\right)_+\in \unitr{A}\subset X_1. 
\]
\end{lem}
\begin{proof} A straightforward direct calculation gives 
\[
\vect{w}\cdot T_{\vect{u},\vect{v}}=\left( d(\vect{u},\vect{w}),\frac{d(\vect{v},\vect{w})}{d(\vect{u},\vect{v})}\right)
\in U_2.
\]
\end{proof}
\begin{cor}\label{cor:phi} Let $A$ be a commutative ring and let $(\vect{u},\vect{v},\vect{w})\in \gnu{3}$.
\begin{enumerate}
\item For all $a,b,c\in \unitr{A}$ we have 
\[
\phi(\vect{u}\cdot a,\vect{v}\cdot b,\vect{w}\cdot c)=\phi(\vect{u},\vect{v},\vect{w})\cdot a^2.
\]
\item For all $X\in \gl{2}{A}$ we have 
\[
\phi\left((\vect{u},\vect{v},\vect{w})\cdot X\right)=\phi(\vect{u},\vect{v},\vect{w})\cdot \mathrm{det}(X).
\] 
\end{enumerate}
\end{cor}

Now, for $n\geq 1$, let 
\[
Y_n=Y_n(A):= \pset{(y_1,\ldots,y_n)\in (\unitr{A})^n}{y_i-y_j\in \unitr{A}\ \forall \ i\not= j}.
\]
We will consider $Y_n$ as a right $\unitr{A}$-set via $(y_1,\ldots,y_n)\cdot a:= (y_1a,\ldots,y_na)$.

For $n\geq 3$, let $\Phi_n:\gnu{n}\to (\unitr{A})^n$ be the map 
\[
\Phi_n(\vect{u}_1,\ldots,\vect{u}_n):= 
\left( \phi(\vect{u}_1,\vect{u}_2,\vect{u}_3),\ldots,\phi(\vect{u}_1,\vect{u}_2,\vect{u}_i), \ldots, 
\phi(\vect{u}_1,\vect{u}_2,\vect{u}_n)\right).
\]

\begin{lem}
For all $(\vect{u}_1,\ldots,\vect{u}_n)\in \gnu{n}$, we have 
 $\Phi_n(\vect{u}_1,\ldots,\vect{u}_n)\in Y_{n-2}$.
\end{lem}
\begin{proof}
 For $3\leq i\leq n$, let 
$y_i:=\phi(\vect{u}_1,\vect{u}_2,\vect{u}_i)$. Let $3\leq i<j\leq n$. Then 
\[
\col{\vect{u}_i}{\vect{u}_j}\in\gl{2}{A} \imp \col{\vect{u}_i}{\vect{u}_j}\cdot T_{\vect{u}_1,\vect{u}_2}\in\gl{2}{A}. 
\]
But 
\[
\col{\vect{u}_i}{\vect{u}_j}\cdot T_{\vect{u}_1,\vect{u}_2}=\matr{y_i}{1}{y_j}{1}\matr{a}{0}{0}{b}
\]
for some $a,b\in\unitr{A}$. On taking the determinant, it follows that $y_i-y_j\in \unitr{A}$. 
\end{proof}

From Corollary \ref{cor:phi} we immediately deduce:

\begin{lem} Let $(\vect{u}_1,\ldots,\vect{u}_n)\in \gnu{n}$. 
\begin{enumerate}
\item For all $a_1,\ldots, a_n\in \unitr{A}$ we have 
\[
\Phi_n(\vect{u}_1\cdot a_1,\ldots,\vect{u}_n\cdot a_n)=\Phi_n(\vect{u}_1,\ldots,\vect{u}_n)\cdot a_1^2.
\]
\item For all $X\in \gl{2}{A}$ we have 
\[
\Phi_n\left((\vect{u}_1,\ldots,\vect{u}_n)\cdot X\right)=\Phi_n(\vect{u}_1,\ldots,\vect{u}_n)\cdot \mathrm{det}(X).
\]
\end{enumerate}
\end{lem}

It follows that $\Phi_n$ induces a well-defined map of orbit sets (which we will continue to denote $\Phi_n$)
\[
X_n/\spl{2}{A}\to Y_{n-2}/(\unitr{A})^2.
\]
Furthermore, this is a map of right $\sq{A}$-sets (noting that the matrix $X\in \gl{2}{A}$ acts via the square class 
of $\mathrm{det}(X)$ on the left).  

\begin{prop}
For all $n\geq 3$, $\Phi_n$ induces a bijection of $\sq{A}$-sets
\[
X_n/\spl{2}{A}\leftrightarrow Y_{n-2}/(\unitr{A})^2.
\] 
\end{prop}
\begin{proof}
Let $\Psi_n:Y_{n-2}\to \gnu{n}$ be the map 
\[
(y_3,\ldots,y_n)\mapsto \left( (0,-1),(1,0),(y_3,1),\ldots,(y_n,1)\right).
\]

Then $\Phi_n\circ\Psi_n=\id{Y_{n-2}}$ since $T_{(0,-1),(1,0)}$ is the identity matrix. 

Now let $\bar{\Psi}_n$ be the induced map from $Y_{n-2}$ to $X_n$,  given by the formula
\[
(y_3,\ldots,y_n)\mapsto \left( 0_+,0_-,(y_3)_+,\ldots,(y_n)_+\right).
\]

For any $y\in A$ and $a\in \unitr{A}$ we have
\[
[y ,1]\matr{a}{0}{0}{a^{-1}}=[ya,a^{-1}]=[ya^2,1]=(ya^2)_+\mbox{ in } X_1.
\] 

It follows that $\bar{\Psi}_n$ induces a well-defined map $Y_{n-2}/(\unitr{A})^2\to X_n/\spl{2}{A}$ satsifying
$\Phi_n\circ \bar{\Psi}_n=\id{Y_{n-2}/(\unitr{A})^2}$. 

It remains to show that $\bar{\Psi}_n:Y_{n-2}/(\unitr{A})^2\to X_n/\spl{2}{A}$ is surjective: If 
$(\vect{u}_1,\ldots,\vect{u}_n)\in \gnu{n}$ then in $X_n/\spl{2}{A}$ we have 
\begin{eqnarray*}
(\bar{\vect{u}}_1,\ldots,\bar{\vect{u}}_n)= (\bar{\vect{u}}_1,\ldots,\bar{\vect{u}}_n)\cdot T_{\vect{u}_1,\vect{u}_2}
= (0_+,0_-,(y_3)_+,\ldots,(y_n)_+)= \bar{\Psi}_n(y_3,\ldots,y_n)
\end{eqnarray*}
where $y_i=\phi(\vect{u}_1,\vect{u}_2,\vect{u}_i)$ for $i\geq 3$.
\end{proof}

Taking the quotient set for the action of $\sq{A}$ on both sides, we deduce:
\begin{cor} For $n\geq 3$,  $\Phi_n$ induces a natural bijection
\[
X_n/\gl{2}{A}\leftrightarrow Y_{n-2}/\unitr{A}.
\]
\end{cor}

For $n\geq 1$, we let 
\[
Z_n=Z_n(A):=\pset{ (z_1,\ldots,z_n)\in \wn{A}^n}{z_i/z_j\in \wn{A}\ \forall i\not= j}
\]
and we let $Z_0=Z_0(A):=\sset{1}$.

We observe that for $n\geq 1$ there is a natural bijection of $\unitr{A}$-sets 
\[
Y_n\leftrightarrow \unitr{A}\times Z_{n-1},\ (y_1,\ldots,y_n)\leftrightarrow 
\left(y_1,\left(\frac{y_2}{y_1},\ldots,\frac{y_n}{y_1}\right)\right)
\]
(where $\unitr{A}$ acts on the first factor of the right-hand side). Thus, taking quotient sets  for the action 
of $(\unitr{A})^2$ and $\unitr{A}$, we obtain:

\begin{cor}
For all $n\geq 3$ we have natural bijections
\begin{eqnarray*}
X_n/\spl{2}{A}&\leftrightarrow &\sq{A}\times Z_{n-3}\\
X_n/\gl{2}{A}&\leftrightarrow & Z_{n-3}\\
\end{eqnarray*}
\end{cor}

\begin{rem} Retracing our steps above, an explicit formula for the first bijection is 
\begin{eqnarray*}
(\bar{\vect{u}}_1,\ldots,\bar{\vect{u}}_n)&\mapsto &
\left( \an{\phi(\vect{u}_1,\vect{u}_2,\vect{u}_3)},\left(
\frac{\phi(\vect{u}_1,\vect{u}_2,\vect{u}_4)}{\phi(\vect{u}_1,\vect{u}_2,\vect{u}_3)},\ldots,
\frac{\phi(\vect{u}_1,\vect{u}_2,\vect{u}_n)}{\phi(\vect{u}_1,\vect{u}_2,\vect{u}_3)}\right)\right)\\
&=& 
\left( \an{\frac{d(\vect{u}_1,\vect{u}_3)d(\vect{u}_1,\vect{u}_2)}{d(\vect{u}_2,\vect{u}_3)}},
\left(\frac{d(\vect{u}_1,\vect{u}_4)d(\vect{u}_2,\vect{u}_3)}{d(\vect{u}_2,\vect{u}_4)d(\vect{u}_1,\vect{u}_3)},
\ldots,\frac{d(\vect{u}_1,\vect{u}_n)d(\vect{u}_2,\vect{u}_3)}{d(\vect{u}_2,\vect{u}_n)d(\vect{u}_1,\vect{u}_3)}
\right)\right)\\
\end{eqnarray*}
and hence the formula for the second is 
\[
(\bar{\vect{u}}_1,\ldots,\bar{\vect{u}}_n) \mapsto 
\left(\frac{d(\vect{u}_1,\vect{u}_4)d(\vect{u}_2,\vect{u}_3)}{d(\vect{u}_2,\vect{u}_4)d(\vect{u}_1,\vect{u}_3)},
\ldots,\frac{d(\vect{u}_1,\vect{u}_n)d(\vect{u}_2,\vect{u}_3)}{d(\vect{u}_2,\vect{u}_n)d(\vect{u}_1,\vect{u}_3)}
\right).
\]
\end{rem}
\begin{rem} If $A=F$ is a field, then clearly $X_1=(F^{2}\setminus\sset{0})/F^\times=\projl{F}$ 
and more generally $X_n$ is naturally the set  of $n$-tuples of \emph{distinct} points of $\projl{F}$. On the other 
hand, $\wn{F}=F^\times\setminus\sset{1}=\projl{F}\setminus\{ \infty, 0,1\}$
 and $Z_n$ consists of $n$-tuples of distinct points of $F^\times\setminus\sset{1}$.

The point 
$x\in F$ is identified with the point of $\projl{F}$ represented by $(x,1)$. Since $d((x,1),(y,1))=x-y$, the 
bijection $X_n/\gl{2}{F}\leftrightarrow Z_{n-3}$ is thus given by the formula
\[
(x_1,\ldots,x_n)\mapsto (\sset{x_1:x_2:x_3:x_4},\ldots,\sset{x_1:x_2:x_3:x_n})
\]  
where 
\[
\sset{x_1:x_2:x_3:x_4}=\frac{(x_1-x_4)(x_2-x_3)}{(x_1-x_2)(x_3-x_4)}
\]
is the classic cross ratio. 
\end{rem}

\begin{cor}\label{cor:zxn} 
For all $n\geq 3$ there are natural isomorphisms of $\sgr{A}$-modules
\[
\Z[X_n]_{\spl{2}{A}}\cong \sgr{A}[Z_{n-3}]
\]
and natural isomorphisms of $\Z$-modules
\[
\Z[X_n]_{\gl{2}{A}}\cong \Z[Z_{n-3}].
\]
\end{cor}
\begin{proof}
If $G$ is a group and if $X$ is a right $G$-set, then for any ring $R$ there is a natural isomorphism
\[
R[X]_G\cong R[X/G],\ \bar{x}\mapsto \bar{x}. 
\]
Thus, for $n\geq 3$,  
\[
\Z[X_n]_{\spl{2}{A}}\cong \Z[X_n/\spl{2}{A}]\cong \Z[\sq{A}\times Z_{n-3}]\cong \Z[\sq{A}][Z_{n-3}]=\sgr{A}[Z_{n-3}].
\]
\end{proof}

For $n\geq 1$, let $\delta_n: \Z[X_{n+1}]\to \Z[X_{n}]$ be the simplicial boundary map
\[
(\bar{\vect{u}}_1,\ldots,\bar{\vect{u}}_{n+1})\mapsto \sum_{i=1}^{n+1}(-1)^{i+1} (\bar{\vect{u}}_1,\ldots,
\widehat{\bar{\vect{u}}_i},\ldots,\bar{\vect{u}}_{n+1})
\]
and let $\ppb{A}:=\coker{\delta_4}$. Note that $(Z[X_n],\delta_n)$ is a complex of $\gl{2}{A}$-modules and that 
$\ppb{A}$ is thus also a $\gl{2}{A}$-module.
\begin{prop}\label{prop:rpb}
 For any commutative ring $A$, $\rpb{A}\cong \ppb{A}_{\spl{2}{A}}$ as $\sgr{A}$-modules, and 
$\pb{A}\cong \ppb{A}_{\gl{2}{A}}$ as $\Z$-modules.
\end{prop}
\begin{proof}
By right exactness of coinvariants, $\ppb{A}_{\spl{2}{A}}$ is naturally identified with the cokernel of the 
map $\bar{\delta}_4:\Z[X_5]_{\spl{2}{A}}\to\Z[X_4]_{\spl{2}{A}}$ of $\sgr{A}$-modules induced by $\delta_4$. Now 
\[
\Z[X_5]_{\spl{2}{A}}\cong \sgr{A}[Z_2]\mbox{ and } \Z[X_4]_{\spl{2}{A}}\cong \sgr{A}[Z_1],
\] 
and, under these identifications, the map $\bar{\delta}_4$ is described as follows: $(z_1,z_2)\in \sgr{A}[Z_2]$ 
corresponds to $(1,z_1,z_2)\in Y_3/(\unitr{A})^2$ and this in turn corresponds to the element
$(0_+,0_-,1_+,(z_1)_+,(z_2)_+)\in \Z[X_5]_{\spl{2}{A}}$. The image of this under $\bar{\delta}_4$ is 
\[
(0_-,1_+,(z_1)_+,(z_2)_+)-(0_+,1_+,(z_1)_+,(z_2)_+)+(0_+,0_-,(z_1)_+,(z_2)_+)-(0_+,0_-,1_+,(z_2)_+)+
(0_+,0_-,1_+,(z_1)_+,(z_2)_+)
\] 
 in $\Z[X_4]_{\spl{2}{A}}$. Recalling  that $(\bar{\vect{u}}_1,\ldots,\bar{\vect{u}}_4)\in \Z[X_4]_{\spl{2}{A}}$ corresponds 
to  
\[
\an{\frac{d(\vect{u}_1,\vect{u}_3)d(\vect{u}_1,\vect{u}_2)}{d(\vect{u}_2,\vect{u}_3)}}
\left( \frac{d(\vect{u}_1,\vect{u}_4)d(\vect{u}_2,\vect{u}_3)}{d(\vect{u}_2,\vect{u}_4)d(\vect{u}_1,\vect{u}_3)}\right)
\in \sgr{A}[Z_1]
\] 
and observing that $d(a_+,b_+)=a-b$ and $d(0_-,a_+)=1$ for all $a\not= b\in A$, we see that 
\[
\bar{\delta}_4(z_1,z_2)=\an{1-z_1}\left(\frac{1-z_1}{1-z_2}\right)-\an{z_1^{-1}-1}\left(\frac{1-z_1^{-1}}{1-z_2^{-1}}\right)
+\an{z_1}\left(\frac{z_2}{z_1}\right)-(z_2)+(z_1)\in \sgr{A}[Z_1].
\]
 Thus the map $\sgr{A}[Z_1]\to\rpb{A}$, $(z)\mapsto \gpb{z}$ induces an isomorphism 
\[
\coker{\bar{\delta}_4}\cong \rpb{A}.
\]
\end{proof}

\begin{rem} We will call the (composite) map
\[
\Z[X_4]\to \sgr{A}[Z_1]\to \rpb{A}, \ 
(\bar{\vect{u}}_1,\ldots,\bar{\vect{u}}_4)\mapsto 
\an{\frac{d(\vect{u}_1,\vect{u}_3)d(\vect{u}_1,\vect{u}_2)}{d(\vect{u}_2,\vect{u}_3)}}
\left( \frac{d(\vect{u}_1,\vect{u}_4)d(\vect{u}_2,\vect{u}_3)}{d(\vect{u}_2,\vect{u}_4)d(\vect{u}_1,\vect{u}_3)}\right)
\]
the \emph{refined cross ratio map}, and will denote it by $\rcr$. In the special case where
$\vect{u}_i=\iota_+(x_i)$ for $x_i\in A$, it takes the form 
\[
(x_1,x_2,x_3,x_4)\mapsto \an{\frac{(x_1-x_3)(x_1-x_2)}{x_2-x_3}}\gpb{\frac{(x_1-x_4)(x_2-x_3)}{(x_1-x_2)(x_3-x_4)}}.
\] 
\end{rem}
\subsection{The isomorphism $\ho{n}{T}{\Z}\cong\ho{n}{B}{\Z}$}
In order to prove Proposition \ref{prop:t2b} below, we follow the strategy of Suslin's proof of 
Theorem 1.8 in \cite{sus:homgln}.

\begin{lem}(\cite[Lemma 1.1]{sus:homgln} )\label{lem:susgln}
Suppose that $\phi_1,\ldots,\phi_m:k\to F$ are field embeddings such that for any $x\in k^\times$ we have
$\prod_{i=1}^m\phi_i(x)=1$. Then $k$ is a finite field of order $p^f$ with $m\geq (p-1)\cdot f$. 
\end{lem}
\begin{rem} This simple but useful result can be extended in many directions. For example:
(See \cite[Lemma 2.2.4]{knudson:book})
Let $A$ be a ring with many units. Let $B$ be any ring. For any $m\geq 1$, there do not exist  ring homomorphisms 
$\phi_1,\ldots,\phi_m:A\to B$ satisfying $\prod_{i=1}^m\phi_i(x)=1$ for all $x\in \unitr{A}$.

Local rings with infinite residue fields are rings with many units, but we will want to include the case of local rings 
with finite residue field below. 
\end{rem}
\begin{cor}\label{cor:susgln}
Suppose that $\phi_1,\ldots,\phi_m:k\to F$ are field embeddings such that for any $x\in k^\times$ we have
$\prod_{i=1}^m\phi_i(x^r)=1$. Then $k$ is a finite field of order $p^f$ and $mr = (p-1)\cdot t$ 
for some $t\geq f$. 
\end{cor}
\begin{proof}
We have $1= \prod_{i=1}^m\phi_i(x^r)=\prod_{i=1}^m\phi_i(x)^r:=\prod_{i=1}^{mr}\psi_i(x)$, and thus $k$ is finite of 
characteristic $p>0$ and $mr\geq (p-1)f$ by Lemma \ref{lem:susgln}. 

On the other hand, if $a\in \F{p}\subset k$ is 
a primitive root modulo $p$, then $1=\prod_i \phi_i(a^r)=a^{mr}$ and thus $p-1| mr$.    
\end{proof}

%\begin{lem}(See \cite[Lemma 2.2.4]{knudson:book})\label{lem:many}
%Let $A$ be a ring with many units. Let $B$ be any ring. For any $m\geq 1$, there do not exist  ring homomorphisms 
%$\phi_1,\ldots,\phi_m:A\to B$ satisfying $\prod_{i=1}^m\phi_i(x)=1$ for all $x\in \unitr{A}$. 
%\end{lem}

\begin{cor}\label{cor:many}
Let $A$ be a local ring with maximal ideal $\mathcal{M}$ and residue field $k$. Suppose that $r\geq 1$ and 
$\phi_1,\ldots,\phi_m:A\to F$ are homomorphisms from $A$ to the field $F$ satisfying 
$\prod_{i=1}^m\phi_i(u^r)=1$ for all $u\in\unitr{A}$. Then $k$ is a finite field with $p^f$ elements and 
$mr=(p-1)t$ where $t\geq f$.  In particular, $(p-1)f\leq mr$.
\end{cor}

\begin{proof} $F$ must have positive characteristic, for otherwise we can choose $1<n\in\unitr{A}\cap \Z$, and the 
hypothesis gives $n^{mr}=1$ in $F$. 

Let $\mathrm{char}(F)=p>0$. Replacing $A$ by $A/pA$ if necessary, we can assume that $A$ is an $\F{p}$-algebra. We 
complete the proof by showing that $\mathcal{M}\subset \ker{\phi_i}$ for all $i$ (and hence that the $\phi_i$ factor 
through $k$):

Let $x\in \mathcal{M}$. For $i=1,\ldots, m$, let $x_i=\phi_i(x)\in F$. If $f(T)\in \F{p}[T]$ satisfies
 $f(0)\not=0$, then $f(x)\in\unitr{A}$. In this case we have 
\[
1=\prod_{i=1}^m\phi_i(f(x))^r=\prod_{i=1}^mf(x_i)^r.
\]

Thus, let $I$ be the ideal of $\F{p}[T_1,\ldots, T_m]$ generated by the set
\[
\{ \left(\prod_{i=1}^mf(T_i)^r\right) -1  \ |\ f(T)\in \F{p}[T]\mbox{ with }f(0)\not= 0\}.
\]
 Let $V$ be the corresponding variety.  Then $(x_1,\ldots,x_m)\in V(F)$.

We observe that $(0,\ldots,0)\in V$ if and only if $p-1| mr$. 

On the other hand, suppose that $(a_1,\ldots,a_m)\in \bF{p}^m$ is algebraic and that $a_j\not=0$ for some $j$. Then there 
exists $f(T)\in\F{p}[T]$ with $f(0)\not=0$ and $f(a_j)=0$. It follows that $\prod_if(a_i)^r=0$ and hence 
$(a_1,\ldots,a_m)\not\in V(\bF{p})$.  Thus 
\[
V(\bF{p})=\left\{
\begin{array}{ll}
\{ 0\}, & p-1|mr\\
\emptyset,& \mbox{ otherwise.}
\end{array}
\right.
\]  
It follows from the Nullstellensatz that the ideal, $J$, of $V$ in $\bF{p}[T_1,\ldots, T_m]$ is given by 
\[
J= \left\{
\begin{array}{ll}
\an{ T_1,\ldots, T_m}, & p-1|mr\\
\bF{p}[T_1,\ldots,T_m],& \mbox{ otherwise.}
\end{array}
\right.
\]
and hence, for \emph{any} field $K$ we have 
\[
V(K)=\left\{
\begin{array}{ll}
\{ 0\}, & p-1|mr\\
\emptyset,& \mbox{ otherwise.}
\end{array}
\right.
\]

Since $(x_1,\ldots,x_m)\in V(F)$, it follows that $p-1|mr$ and $x_i=\phi_i(x)=0$ for all $i$.
\end{proof}

For $r\geq 1$, we denote by $A(r)$ the $\Z[\unitr{A}]$-module obtained by 
making $u\in\unitr{A}$ act on $A$ as multiplication by $u^r$.  

\begin{lem}\label{lem:sus2} 
Let $m, r\geq 1$. Let $n_1,\ldots, n_k$ satisfy $n_1+\cdots n_k=m$ and $n_i\geq 1$.
Let $A$ be a local ring with residue field $k$. If $k$ is finite of order $p^f$ we suppose that 
$mr<(p-1)f$.

 Let $T^n(A(r))$ denote 
either $\Extpow{n}{\Z}{A(r)}$ or $\sym{n}{\Z}{A(r)}$, considered as $\unitr{A}$ modules with the diagonal action. 

Then
\[
\ho{i}{\unitr{A}}{T^{n_1}(A(r))\otimes \cdots \otimes T^{n_k}(A(r))}=0
\] 
for all $i\geq 0$. 
\end{lem}
\begin{proof}
This follows from Corollary \ref{cor:many} by the same argument verbatim as that by which Suslin proves 
Corollary 1.6 from Lemma 1.1 in \cite{sus:homgln}. 
\end{proof}

\begin{lem} \label{lem:sus}
Let $m, r\geq 1$.
Let $A$ be a local integral domain with residue field $k$. If $k$ is finite of order $p^f$ we suppose that 
$mr<(p-1)f$.

For all $i\geq 0$ we have 
\[
\ho{i}{\unitr{A}}{\ho{m}{A(r)}{\Z}}=0.
\] 
\end{lem}
\begin{proof}
If $\mathrm{char}(A)=0$, then $\ho{m}{A(r)}{\Z}=\Extpow{m}{\Z}{A(r)}$ and the statement follows at once from
Lemma \ref{lem:sus2}.

Otherwise $A$ is an $\F{p}$-algebra for some $p>0$. Then $\ho{m}{A(r)}{\F{p}}$ is a direct sum of modules of the 
form $\Extpow{s}{\Z}{A(r)}\otimes \sym{t}{\Z}{A(r)}$ with $s+t\leq m$. It follows from Lemma \ref{lem:sus2} that 
$\ho{i}{\unitr{A}}{\ho{m}{A(r)}{\F{p}}}=0$. 

On the other hand, the short exact sequence 
\[
\xymatrix{
0\ar[r]
& \Z\ar^-{p}[r]
& \Z\ar[r]
&\Z/p\Z=\F{p}
\ar[r]
& 0
}
\]
induces a long exact homology sequence for $\ho{\bullet}{A(r)}{}$, which decomposes into short exact sequences 
\[
0\to \ho{k}{A(r)}{\Z}\to \ho{k}{A(r)}{\F{p}}\to \ho{k-1}{A(r)}{\Z}\to 0 \mbox{ ($k\geq 2$)}
\]  
and an isomorphism
\[
\ho{1}{A(r)}{\Z}\cong A(r)\cong \ho{1}{A(r)}{\F{p}}.
\]

The vanishing of $\ho{i}{\unitr{A}}{\ho{m}{A(r)}{\Z}}$ for all $i\geq 0$ 
then follows from a straightforward induction on $m$.
\end{proof}

We let $T=T(A)$ denote the subgroup of $\spl{2}{A}$ consisting of diagonal matrices:
\[
T(A):=\left\{ \matr{u}{0}{0}{u^{-1}}\ |\ u\in\unitr{A}\right\}.
\]
Thus $T(A)\cong \unitr{A}$.  We let $B=B(A)$ denote the subgroup consisting of lower triangular matrices:
 \[
B(A):=\left\{ \matr{u}{a}{0}{u^{-1}}\ |\ u\in\unitr{A}, a\in A \right\}.
\]
Thus there is natural (split) group extension 
\begin{eqnarray}\label{eqn:ubt}
1\to V\to B\to T\to 1
\end{eqnarray}
where 
\[
V=T(A):=\left\{ \matr{1}{a}{0}{1}\ |\ a\in A \right\}\cong A.
\]

Here $T\cong \unitr{A}$ acts on $V\cong A$ by conjugation. With the given identifications, $u\in \unitr{A}$ acts on 
$a\in A$ as multiplication by $u^2$. Thus $V\cong A(2)$ as a $\Z[\unitr{A}]$-module.

\begin{prop}\label{prop:t2b} 
Let $n \geq 1$.
Let $A$ be a local integral domain with residue field $k$. If $k$ is finite of order $p^f$ we suppose that 
$(p-1)f>2n$.

The natural maps $B\to T$ and $T\to B$  induce isomorphisms on homology 
\[
\ho{n}{T}{\Z}\cong \ho{n}{B}{\Z}
\]
\end{prop}
\begin{proof}
The Hochschild-Serre spectral sequence associated to the extension (\ref{eqn:ubt}) takes the form
\[
E^2_{i,j}=\ho{i}{T}{\ho{j}{V}{\Z}}=\ho{i}{\unitr{A}}{\ho{j}{A(2)}{\Z}}\Rightarrow \ho{i+j}{B}{Z}.
\]

By Lemma \ref{lem:sus}, it follows that $E^2_{i,j}=0$ if $0<j\leq n$,  
and $E^2_{i,0}=\ho{i}{T}{\Z}$ for all $i$.

Hence $\ho{n}{B}{\Z}=\ho{n}{T}{\Z}$.
\end{proof}
\begin{rem} \label{rem:suff}
In particular,  $\ho{n}{T}{\Z}\cong\ho{n}{B}{\Z}$ for all $n\leq 3$ provided \\ 
\[
\card{k}\not\in \{ p^f\ |\ p \mbox{ prime and } (p-1)f\leq 6\}=\{2,3,4,5,7,8,9,16,27,32, 64 \}. 
\]

In the remainder of the paper we will say that a field $k$ is \emph{sufficiently large} 
if either $k$ is infinite or if $\card{k}=p^f$ where $(p-1)f>6$. 
\end{rem}
\subsection{The complex $L_\bullet$}  For a commutative ring $A$, we let $L_n=L_n(A):=\Z[X_{n+1}]$. Equipped with the 
boundary $\delta_n:L_n\to L_{n-1}$ this yields a complex, $L_\bullet$, 
 of right $\gl{2}{A}$-modules. Restricting the group action, 
this is also a complex of $\spl{2}{A}$-modules.
% We let $\epsilon:C_0\to 
%\Z$ denote the map $x\mapsto 1$ for $x\in X_1$.  We will regard $\Z$ as a complex of $\gl{2}{A}$-modules (with the 
%trivial action) concentrated in dimension $0$ and $\epsilon$ as a morphism of complexes.

We now restrict attention to the case where $A$ is a commutative local ring with residue field $k$. We let 
$\pi:A\to k$ denote the canonical surjective quotient map. So $\unitr{A}= \pi^{-1}(k^\times)$. More generally, if 
$X\in M_n(A)$ is an 
$n\times n$ matrix with coefficients in $A$, we let $\pi(X)\in M_n(k)$ denote the matrix obtained by applying $\pi$ to 
each entry of $X$. Since $A$ is a local ring $X\in \gl{n}{A}$ if and only if $\pi(X)\in \gl{n}{k}$.  
 
Similarly, $\vect{u}=(u_1,u_2)\in U_2(A)$ if and only if $\pi(\vect{u})\in U_2(k)=k^2\setminus\{ 0\}$ and 
$\bar{\vect{u}}\in X_1(A)$ if and only if $\overline{\pi(\vect{u})}\in X_1(k)=\projl{k}$. Furthermore,
\[
(\vect{u}_1,\ldots,\vect{u}_n)\in \genu{n}{A}\ \iff\ 
(\pi(\vect{u}_1),\ldots,\pi(\vect{u}_n))\in \genu{n}{k}
\]
and hence 
\[
(\bar{\vect{u}}_1,\ldots,\bar{\vect{u}}_n)\in X_n(A)\ \iff\ 
(\overline{\pi(\vect{u}_1)},\ldots,\overline{\pi(\vect{u}_n)})\in X_n(k).
\]

\begin{lem} 
 $H_n(L_\bullet)=0$ for $1\leq n< \card{k}$.
\end{lem}
\begin{proof} When $A=k$ is a field, the argument is given in \cite{hut:cplx13}, Lemma 4.4. This argument is easily 
adapted to the current situation as follows: 

For any subset $S$ of $\projl{k}$, let $D_n(S)$ denote the subgroup of $L_n(A)$ generated by those $(n+1)$-tuples
$(\bar{\vect{u}}_1,\ldots,\bar{\vect{u}}_{n+1})\in X_{n+1}(A)$  which satisfy 
$S\subset \{ \overline{\pi(\vect{u}_1)}, \ldots,\overline{\pi(\vect{u}_{n+1})}\}$. Thus $D_n(S)=0$ if $\card{S}>n+1$.
Furthermore, $D_n(S_1\cup S_2)=D_n(S_1)\cap D_n(S_2)$ for any $S_1,S_2\subset \projl{k}$. 

Now for each $x\in \projl{k}$, choose $\vect{u}_x\in U_2(A)$ satisfying $\overline{\pi(\vect{u}_x)}=x$ and  for $n\geq 0$
define a homomorphism $S_x:L_n\to L_{n+1}$ by 
\[
S_x(\bar{\vect{u}}_1,\ldots,\bar{\vect{u}}_{n+1})=
\left\{
\begin{array}{ll}
(\bar{\vect{u}}_x,\bar{\vect{u}}_1,\ldots,\bar{\vect{u}}_{n+1}),& x\not\in 
\{ \overline{\pi(\vect{u}_1)}, \ldots,\overline{\pi(\vect{u}_{n+1})}\}\\
0,& \mbox{otherwise}\\
\end{array}
\right.
\]

Thus if $(\bar{\vect{u}}_1,\ldots,\bar{\vect{u}}_{n+1})\in X_{n+1}(A)$ and if  
$x\not\in 
\{ \overline{\pi(\vect{u}_1)}, \ldots,\overline{\pi(\vect{u}_{n+1})}\}$ then 
\[
\delta S_x(\bar{\vect{u}}_1,\ldots,\bar{\vect{u}}_{n+1})=
(\bar{\vect{u}}_1,\ldots,\bar{\vect{u}}_{n+1})-S_x\delta(\bar{\vect{u}}_1,\ldots,\bar{\vect{u}}_{n+1}).
\]
On the other hand, if $x= \overline{\pi(\vect{u}_j)}$ for some $j$ then 
\[
S_x\delta(\bar{\vect{u}}_1,\ldots,\bar{\vect{u}}_{n+1})=
(-1)^{j+1}(\bar{\vect{u}}_x,\bar{\vect{u}}_1,\ldots,\widehat{\bar{\vect{u}}_j},\ldots,\bar{\vect{u}}_{n+1}).
\]
and hence 
\begin{eqnarray*}
0&=&\delta S_x(\bar{\vect{u}}_1,\ldots,\bar{\vect{u}}_{n+1})\\
&=&
(\bar{\vect{u}}_1,\ldots,\bar{\vect{u}}_{n+1})-S_x\delta(\bar{\vect{u}}_1,\ldots,\bar{\vect{u}}_{n+1})
-\left\{ S_x\delta(\bar{\vect{u}}_1,\ldots,\bar{\vect{u}}_{n+1})-(-1)^j
(\bar{\vect{u}}_x,\bar{\vect{u}}_1,\ldots,\widetilde{\bar{\vect{u}}_j},\ldots,\bar{\vect{u}}_{n+1})\right\}.
\end{eqnarray*}
In either case we have 
\[
\delta S_x(\bar{\vect{u}}_1,\ldots,\bar{\vect{u}}_{n+1})=
(\bar{\vect{u}}_1,\ldots,\bar{\vect{u}}_{n+1})-S_x\delta(\bar{\vect{u}}_1,\ldots,\bar{\vect{u}}_{n+1})
+w
\]
where $w\in D_n(\{ x\})$. Furthemore, if $(\bar{\vect{u}}_1,\ldots,\bar{\vect{u}}_{n+1})\in D_n(S)$ for some 
subset $S$ of $\projl{k}$ then $w\in D_n(S\cup \{ x\})$.

Suppose now that $1\leq n < \card{k}$ and that $x_1,\ldots,x_{n+2}$ are $n+2$ distinct points of $\projl{k}$. 
Let $z\in L_n(A)$ be a cycle. Then
\[
(\delta S_{x_1}-\mathrm{Id})z=S_{x_1}\delta(z)+z_1=z_1
\] 
where $z_1\in D_n(\{ x_1\})$ and $z_1$ is again a cycle.

Thus $(\delta S_{x_2}-\mathrm{Id})z_1=z_2$ where $z_2$ is a cycle belonging to $D_n(\{ x_1,x_2\})$. Repeating the 
process we get
\[
(\delta S_{x_{n+2}}-\mathrm{Id})(\delta S_{x_{n+1}}-\mathrm{Id})\cdots (\delta S_{x_1}-\mathrm{Id})z\in 
D_n(\{ x_1,\ldots,x_{n+2}\})=0.
\]
This equation has the form $\delta(y)+(-1)^{n+2}z=0$ and hence $z=\delta((-1)^{n+1}y)$ is a boundary as required.
\end{proof}

\subsection{The third homology of $\mathrm{SL}_2$ of local domains}
Recall that we say a field $k$ is \emph{sufficiently large} if either $k$ is infinite or 
$\card{k}=p^f$ with $(p-1)f>6$; see Remark \ref{rem:suff}. 

\begin{thm}\label{thm:h3sl2A}
Let $A$ be a local integral domain with sufficiently large residue field $k$. 
Then 
there is a natural short exact sequence of $\zhalf{\sgr{A}}$-modules 
\[
0\to \mathrm{tor}\zhalf{(\mu_A,\mu_A)}\to \ho{3}{\spl{2}{A}}{\zhalf{\Z}}\to \zhalf{\rbl{A}}\to 0.
\]
\end{thm}
\begin{proof} In the case where $A$ is a field, the proof can be found in \cite[section 4]{hut:cplx13}. 
We indicate here the adaptions needed to extend that proof to the current context:

Associated to the complex $L_\bullet=L_\bullet(A)$ there is hyperhomology spectral sequence of the form 
\[
E^1_{p,q}=\ho{p}{\spl{2}{A}}{\zhalf{L_q}}\Rightarrow \ho{p+q}{\spl{2}{A}}{\zhalf{L_\bullet}}
\]
and furthermore the augmentation $L_0\to \Z$ induces an isomorphism 
\[
\ho{n}{\spl{2}{A}}{\zhalf{L_\bullet}}\cong\ho{n}{\spl{2}{A}}{\zhalf{\Z}}
\]
for $n\leq 3$ by Remark \ref{rem:suff}.

The $\spl{2}{A}$-modules $L_n$ are permutation modules, so the $E^1$-terms are calculated using Shapiro's Lemma:

$\spl{2}{A}$ acts transitively on $X_1$ and the stabilizer of $0_+\in X_1$ is $B=B_A$. Thus 
\[
L_0=\Z[X_1]\cong\Z[B\backslash \spl{2}{A}]
\cong\Ind{{\Z}[B]}{{\Z}[\spl{2}{A}]}{{\Z}}
\]
and hence
\[
E^1_{p,0}=\ho{p}{\spl{2}{A}}{\zhalf{L_0}}\cong \ho{p}{B}{\zhalf{\Z}}.
\]

Similarly, $\spl{2}{A}$ acts transitively on $X_2$ and the stabilizer of $(0_+,0_-)$ is $T=T_A$, so that
\[
E^1_{p,1}=\ho{p}{\spl{2}{A}}{\zhalf{L_1}}\cong \ho{p}{T}{\zhalf{\Z}}.
\]

For $n\geq 3$, the stabilizer in $\spl{2}{A}$ of $(\bar{\vect{u}}_1,\ldots,\bar{\vect{u}}_{n})\in X_n$ 
is $Z(\spl{2}{A})\cong\mu_2(A)$. By Corollary \ref{cor:zxn} it follows that for $q\geq 2$ we have 
\[
E^1_{p,q}=\zhalf{\sgr{A}}[Z_{q-2}]\otimes\ho{p}{\mu_2(A)}{\Z}=
\left\{
\begin{array}{ll}
\zhalf{\sgr{A}}[Z_{q-2}],& p=0\\
0,& p>0\\
\end{array} 
\right.
\]
 where $Z_n=Z_n(A)$ as above.   

Thus our $E^1$-page has the form
\begin{eqnarray*}
\xymatrix{
\vdots&\vdots&\vdots&\vdots&\\
\zhalf{\sgr{A}}[Z_2]\ar[d]^-{d^1}&0 
&\vdots&\vdots&\hdots\\
\zhalf{\sgr{A}}[Z_1]\ar[d]^-{d^1}&0&0&0&\hdots\\
\zhalf{\sgr{A}}\ar[d]^-{d^1}&0&0&0&\hdots\\
\zhalf{\Z}\ar[d]^-{d^1}&\ho{1}{T}{\zhalf{\Z}}\ar[d]^-{d^1}&\ho{2}{T}{\zhalf{\Z}}\ar[d]^-{d^1}&
\ho{3}{T}{\zhalf{\Z}}\ar[d]^-{d^1}&\hdots\\
\zhalf{\Z}&\ho{1}{T}{\zhalf{\Z}}&\ho{2}{T}{\zhalf{\Z}}&\ho{3}{T}{\zhalf{\Z}}&\hdots
}
\end{eqnarray*}

Now $T\cong\unitr{A}$. Thus $E^1_{p,q}\cong \ho{p}{\unitr{A}}{\zhalf{\Z}}$ for $p\leq 3$ and $q\in \{ 0,1\}$. 

Now let 
\[
w:=\matr{0}{-1}{1}{0}\in\spl{2}{A}.
\]
Then $w(0_+)=0_-$ and $w(0_-)= 0_+$. It follows easily that the differential
\[
d^1:E^1_{p,1}=\ho{p}{T}{\zhalf{\Z}}\to \ho{p}{T}{\zhalf{\Z}}= E^1_{p,0}
\]
is the map
\begin{eqnarray*}
\xymatrix{
\ho{p}{T}{\zhalf{\Z}}\ar[r]^-{w_p-1}
&
\ho{p}{T}{\zhalf{\Z}}
}
\end{eqnarray*}
where $w_p:\ho{p}{T}{\zhalf{\Z}}\to\ho{p}{T}{\zhalf{\Z}}$ is the map induced by conjugation by $w$. 
However, conjugating by $w$ is just the inversion map on $\unitr{A}\cong T$.  
Thus $d^1:\zhalf{\Z}=E^1_{0,1}\to E^1_{0,0}=\zhalf{\Z}$ is the zero map. $d^1:\zhalf{\unitr{A}}=E^1_{1,1}\to E^1_{1,0}=
\zhalf{\unitr{A}}$ is the map $u\mapsto u^{-2}$ and hence is an isomorphism. 
$d^1:\Extpow{2}{\Z}{\zhalf{\unitr{A}}}=E^1_{2,1}\to E^1_{2,0}=\Extpow{2}{\Z}{\zhalf{\unitr{A}}}$ is the zero map. 

Finally, $E^1_{3,1}=E^1_{3,0}=\ho{3}{\unitr{A}}{\zhalf{\Z}}\cong \Extpow{3}{\Z}{\zhalf{\unitr{A}}}\oplus 
\mathrm{tor}\zhalf{(\mu_A,\mu_A)}$. The map $d^1:E^1_{3,1}\to E^1_{3,0}$ is an isomorphism of the first factor and the zero 
map on the second factor.

The differential 
\[
d^1:\zhalf{\sgr{A}}\cong \ho{0}{\spl{2}{A}}{L_2}=E^1_{0,2}\to E^1_{0,1}=\ho{0}{\spl{2}{A}}{L_1}\cong \zhalf{\Z}
\]
is the natural augmentation sending $\an{u}$ to $1$ for any $u\in\unitr{A}$. 

As in the proof of 
\cite[Theorem 4.3]{hut:cplx13}, the differential 
\[
d^1:\zhalf{\sgr{A}}[Z_1]\cong \ho{0}{\spl{2}{A}}{L_3}=E^1_{0,3}\to E^1_{0,2}=\ho{0}{\spl{2}{A}}{L_2}\cong \zhalf{\sgr{A}}
\]
is the $\sgr{A}$-homomorphism sending $(z)$ to $\pf{z}\pf{1-z}\in \aug{A}^2$ for any $z\in \wn{A}$.

By the proof of Proposition \ref{prop:rpb} above, the differential 
\[
d^1:\zhalf{\sgr{A}}[Z_2]\cong \ho{0}{\spl{2}{A}}{L_4}=E^1_{0,4}\to E^1_{0,3}=\ho{0}{\spl{2}{A}}{L_3}\cong 
\zhalf{\sgr{A}}[Z_1]
\]
is the map 
\[
(z_1,z_2)\mapsto \an{1-z_1}\left(\frac{1-z_1}{1-z_2}\right)-\an{z_1^{-1}-1}\left(\frac{1-z_1^{-1}}{1-z_2^{-1}}\right)
+\an{z_1}\left(\frac{z_2}{z_1}\right)-(z_2)+(z_1) 
\]

Thus the $E^2$-page of our spectral sequence has the form 
\begin{eqnarray*}
\xymatrix{
\zhalf{\rpbker{A}}&0&0&\vdots\\
\zhalf{\aug{A}}/\zhalf{\mathcal{J}_A}&0&0&\vdots\\
0&0&\Extpow{2}{\Z}{\zhalf{\unitr{A}}}&\vdots\\
\zhalf{\Z}&0&\Extpow{2}{\Z}{\zhalf{\unitr{A}}}&\mathrm{tor}\zhalf{(\mu_A,\mu_A)}
}
\end{eqnarray*}
where $\mathcal{J}_A\subset \sgr{A}$ is the ideal generated by the Steinberg elements $\pf{u}\pf{1-u}$. 

Clearly there are no nonzero $d^2$-differentials. So the $E^3$-page has the form 
\begin{eqnarray*}
\xymatrix{
E^3_{0,4}\ar[dddrr]^-{d^3}&0&0&\vdots\\
\zhalf{\rpbker{A}}\ar[dddrr]^-{d^3}&0&0&\vdots\\
\zhalf{\aug{A}}/\zhalf{\mathcal{J}_A}&0&0&\vdots\\
0&0&\Extpow{2}{\Z}{\zhalf{\unitr{A}}}&\vdots\\
\zhalf{\Z}&0&\Extpow{2}{\Z}{\zhalf{\unitr{A}}}&\mathrm{tor}\zhalf{(\mu_A,\mu_A)}
}
\end{eqnarray*}

The argument now concludes exactly as in \cite[section 4]{hut:cplx13}: The 
cokernel of the differential $d^3:E^3_{0,4}\to \Extpow{2}{\Z}{\zhalf{\unitr{A}}}$ is annihilated by $2$ and hence 
$E^4_{2,1}=E^\infty_{2,1}=0$. There is a commutative diagram 
\begin{eqnarray*}
\xymatrix{
\zhalf{\rpbker{A}}\ar[r]^-{d^3}\ar[d]\ar[dr]^-{\lambda_2}
&
\Extpow{2}{\Z}{\zhalf{\unitr{A}}}\ar[d]^{\cong}\\
\zhalf{\pb{A}}\ar[r]^-{\lambda}
&
\asym{2}{\Z}{\zhalf{\unitr{A}}}
}
\end{eqnarray*}
 It follows that $E^4_{0,3}=E^\infty_{0,3}=\ker{\lambda_2:\zhalf{\rpbker{A}}\to \zhalf{\asym{2}{\Z}{\unitr{A}}}}=
\zhalf{\rbl{A}}$.  This completes the proof of the theorem. 
\end{proof}

\subsection{Local domains with finite residue fields}\label{sec:bwfin}
\begin{cor}\label{cor:bwfin}
Let $A$ be a local integral domain with sufficiently large residue field.  Then 
the natural map $\ho{3}{\spl{2}{A}}{{\Z}}\to \kind{A}$ induces an isomorphism 
\[
\ho{3}{\spl{2}{A}}{\zhalf{\Z}}_{\unitr{A}}\cong \zhalf{\kind{A}}.
\]
\end{cor}

\begin{proof}
Theorem 
\ref{thm:h3sl2A} implies the exactness of the top row in the diagram 
\[
\xymatrix{
0\ar[r]
&\mathrm{tor}\zhalf{(\mu_{A},\mu_{A})}\ar[r]\ar[d]^-{=} 
&\ho{3}{\spl{2}{A}}{\zhalf{\Z}}\ar[r]\ar[d]
& \zhalf{\rbl{A}}\ar[r]\ar[d]^-{=}
&0\\
0\ar[r]
&\mathrm{tor}\zhalf{(\mu_{A},\mu_{A})}\ar[r]
&\zhalf{\kind{A}}\ar[r]
& \zhalf{\bl{A}}\ar[r]
&0\\
}
\]
The bottom row is exact by \cite[Theorem 5.1]{mirzaii:mok2} (infinite residue field) and 
\cite[Theorem 6.1]{mirzaii:arxivbw} (finite residue field). 

The corollary now follows by taking $\unitr{A}$-coinvariants on the top row, since 
$\mathrm{tor}(\mu_A,\mu_A)_{\unitr{A}}= \mathrm{tor}(\mu_A,\mu_A)$ and 
$\zhalf{\rbl{A}}_{\unitr{A}}=\zhalf{\bl{A}}$ by Proposition \ref{prop:rblrpbker} (2). 
\end{proof}
\begin{cor}\label{cor:rpbkerA}
Let $A$ be a local integral domain with sufficiently large residue field.  
Then 
\[
\aug{A}\zhalf{\rpbker{A}}\cong\aug{A}\ho{3}{\spl{2}{R}}{\zhalf{\Z}}=\hot{A}{\zhalf{\Z}}.
\]
\end{cor}
\begin{proof}
Combining Theorem \ref{thm:h3sl2A} with the statements in 
Proposition \ref{prop:mirzaii} and Corollary \ref{cor:bwfin} we obtain a commutative 
diagram with exact rows and columns
\[
\xymatrix{
&
&
0\ar[d]
&
0\ar[d]
&\\
&&
\aug{A}\ho{3}{\spl{2}{A}}{\zhalf{\Z}}\ar[r]\ar[d]
&\aug{A}\zhalf{\rbl{A}}\ar[d]
&\\
0\ar[r]
&
\mathrm{tor}\zhalf{(\mu_A,\mu_A)}\ar[r]\ar^-{=}[d]
&
\ho{3}{\spl{2}{A}}{\zhalf{\Z}}\ar[r]\ar[d]
&
\zhalf{\rbl{A}}\ar[r]\ar[d]
&
0\\
0\ar[r]
&
\mathrm{tor}\zhalf{(\mu_A,\mu_A)}\ar[r]
&
\zhalf{\kind{A}}\ar[r]\ar[d]
&
\zhalf{\bl{A}}\ar[r]\ar[d]
&
0\\
&&
0
&
0
&\\
}
\]
from which it follows that 
\[
\aug{A}\zhalf{\rbl{A}}\cong\aug{A}\ho{3}{\spl{2}{R}}{\zhalf{\Z}}=\hot{A}{\zhalf{\Z}}.
\]
On the other hand, Proposition \ref{prop:rblrpbker} tells us that $\aug{A}\zhalf{\rbl{A}}\cong 
\aug{A}\zhalf{\rpbker{A}}$.
\end{proof}

\section{Submodules and special elements in $\rpb{A}$}\label{sec:rpbA}
In this section, $A$ will denote a commutative local ring, with maximal ideal $\mathcal{M}$ and residue 
field 
$k=A/\mathcal{M}$.  Furthermore, we will suppose that $k$ has at least four elements.

In this case, we have $\unitr{A}=A\setminus \mathcal{M}$ and $\wn{A}=\unitr{A}\setminus U_1$, where 
$U_1=U_{1,A}=1+\mathcal{M}$. In particular,
if $x\in \unitr{A}$, then $x\in \wn{A} \iff x^{-1}\in\wn{A}$. 

\subsection{The modules $\ks{i}{A}$}

As in \cite{hut:rbl11}, we define two families of elements of $\rpb{A}$. 

Given $x\in \wn{A}$ we define 
\[
\suss{1}{x}:=\gpb{x}+\an{-1}{\gpb{x^{-1}}}
\]
and 
\[
\suss{2}{x}:=
\an{1-x}\left(\an{x}\gpb{x}+\gpb{x^{-1}}\right)=\an{x^{-1}-1}\gpb{x}+\an{1-x}\gpb{x^{-1}}.
\]

Observe, from the definitions,  that $\suss{i}{x^{-1}}=\an{-1}\suss{i}{x}$ for all $x\in \wn{A}$. 

\begin{lem} \label{lem:sus1}
For $i=1,2$ we have 
\begin{enumerate}
\item \label{lem:susa}
$\suss{i}{xy}=\an{x}\suss{i}{y}+\suss{i}{x}$ whenever $x,y,xy\in\wn{A}$.
\item $\an{x}\suss{i}{x^{-1}}+\suss{i}{x}=0$ for all $x\in \wn{A}$. 
\end{enumerate}
\end{lem}
\begin{proof}
\begin{enumerate}
\item The proof of Lemma 3.1 in \cite{hut:rbl11} adapts without alteration.
\item  Let $x\in \wn{A}$. Choose $y\in \wn{A}$ such that $xy\in \wn{A}$ also. (Note that this is possible because of 
the hypothesis that $\card{k}\geq 4$.)

Then by part (\ref{lem:susa}) we have 
\[
\suss{i}{xy}=\an{y}\suss{i}{x}+\suss{i}{y} \imp \an{y}\suss{i}{x}=\suss{i}{xy}-\suss{i}{y}
\] 
and 
\[
\suss{i}{y}=\suss{i}{xy\cdot x^{-1}}=\an{xy}\suss{i}{x^{-1}}+\suss{i}{xy}\imp \an{xy}\suss{i}{x^{-1}}=\suss{i}{y}
-\suss{i}{xy}.
\]
Form these we deduce that 
\[
\an{y}\suss{i}{x}=-\an{xy}\suss{i}{x^{-1}}.
\]
Multiplying both sides of this equation by $\an{y}$ gives the result.
\end{enumerate}
\end{proof}

\begin{lem} Let $u\in U_1$. For any $w_1,w_2\in \wn{A}$, we have 
\[
\suss{i}{w_1u}-\an{u}\suss{i}{w_1}=\suss{i}{w_2u}-\an{u}\suss{i}{w_2}
\]
for $i=1,2$.
\end{lem}
\begin{proof}
First suppose that $w_1\not\equiv w_2\pmod{U_1}$.  Then $w_1w_2^{-1}\in\wn{A}$ and hence 
\begin{eqnarray*}
\suss{i}{w_1u}=\suss{i}{(w_2^{-1}w_1)\cdot (w_2u)}\\
&=& \an{w_2u}\suss{i}{w_1w_2^{-1}}+\suss{i}{w_2u}\\
&=& \an{u}\left( \an{w_2}\suss{i}{w_1w_2^{-1}}\right) +\suss{i}{w_2u}\\
&=& \an{u}\left(\suss{i}{w_2}-\suss{i}{w_1}\right)+\suss{i}{w_2u}
\end{eqnarray*}
giving the result in this case.

On the other hand, if $w_1\equiv w_2\pmod{U_1}$ choose $w_3\in \wn{A}$ with $w_3\not\equiv w_1\pmod{U_1}$. Then 
\[
\suss{i}{w_1u}-\an{u}\suss{i}{w_1}=\suss{i}{w_3u}-\an{u}\suss{i}{w_3}\suss{i}{w_2u}-\an{u}\suss{i}{w_2}.
\] 
\end{proof}

We now extend the definition of $\suss{i}{x}$ to allow $x\in U_1$. For $u\in U_1$, we define 
\[
\suss{i}{u}:=\suss{i}{uw}-\an{u}\suss{i}{w}
\]
for any $w\in \wn{A}$.

\begin{prop} For $i=1,2$ the maps $\unitr{A}\to \rpb{A}$, $x\mapsto \suss{i}{x}$ define $1$-cocycles; i.e. we 
have 
\[
\suss{i}{xy}=\an{x}\suss{i}{y}+\suss{i}{x}
\]
for all $x,y\in \unitr{A}$. 
\end{prop}
\begin{proof}
If $x,y, xy \in \wn{A}$, this is part (1) of Lemma \ref{lem:sus1}. 

If $x,y\in \wn{A}$, but $xy\in U_1$, then 
\begin{eqnarray*}
\suss{i}{xy}&=&\suss{i}{xy\cdot y^{-1}}-\an{xy}\suss{i}{y^{-1}}\\
&=& \suss{i}{x}-\an{x}\left(\an{y}\suss{i}{y^{-1}}\right)\\
&=& \suss{i}{x}+\an{x}\suss{i}{y}
\end{eqnarray*}
using Lemma \ref{lem:sus1} (2) in the last step.

If $x\in U_1$ and $y\in \wn{A}$, the identity is just the definition of $\suss{i}{x}$.

On the other hand, if $x\in \wn{A}$ and $y\in U_1$, then choose $w\in \wn{A}$ such that $xw\in \wn{A}$. We have 
\begin{eqnarray*}
\an{x}\suss{i}{y}+\suss{i}{x}&=& \an{x}\left(\suss{i}{wy}-\an{y}\suss{i}{w}\right)+\suss{i}{x}\\
&=& \an{x}\suss{i}{wy}+\suss{i}{x}-\an{xy}\suss{i}{w}\\
&=& \suss{i}{xyw}-\an{x}\suss{i}{yw}\\
&=& \suss{i}{xy}.
\end{eqnarray*}

Finally, suppose that $x,y\in U_1$. Let $w\in \wn{A}$. Then 
\begin{eqnarray*}
\suss{i}{xy}&=& \suss{i}{xyw}-\an{xy}\suss{i}{w}\\
&=& \an{x}\suss{i}{yw}+\suss{i}{x}-\an{xy}\suss{i}{w}\\
&=& \an{x}\left(\suss{i}{yw}-\an{y}\suss{i}{w} \right)+\suss{i}{x}\\
&=& \an{x}\suss{i}{y}+\suss{i}{x}
\end{eqnarray*}
as required.
\end{proof}

We recall here, from \cite{hut:rbl11} some of the basic algebraic properties of the $\suss{i}{x}$. (The proofs given 
in \cite{hut:rbl11} for the case of fields adapt without change to the case of local rings.)

\begin{prop}\label{prop:cocycle}
For $i\in \{1,2\}$ we have:
\begin{enumerate}
\item $\pf{x}\suss{i}{y}=\pf{y}\suss{i}{x}$ for all $x,y$
\item
 $\suss{i}{xy^2}=\suss{i}{x}+\suss{i}{y^2}$ for all $x,y$
\item $\pf{x}\suss{i}{y^2}=0$ for all $x,y$
\item $2\cdot\suss{i}{-1}=0$ for all $i$ 
\item $\suss{i}{x^2}=-\pf{x}\suss{i}{-1}$ for all $x$
\item $2\cdot\suss{i}{x^2}=0 $ for all $x$ and if $-1$ is a square in $\unitr{A}$ 
then $\suss{i}{x^2}=0$ for all $x$.
\item $\pf{x}\pf{y}\suss{i}{-1}=0$ for all $x,y$
\item $\an{-1}\pf{x}\suss{i}{y}=\pf{x}\suss{i}{y}$ for all $x,y$
\item Let 
\[
\epsilon(A):=\left\{
\begin{array}{ll}
1, & -1\in (\unitr{A})^2\\
2,&  -1\not\in (\unitr{A})^2\\
\end{array}
\right.
\]
The map $\unitr{A}/(\unitr{A})^2\to\rpb{A},\an{x}\mapsto \epsilon(A)\suss{i}{x}$ is a well-defined $1$-cocycle.
\end{enumerate}
\end{prop}

\begin{cor}\label{cor:psi-1}
 For $i=1,2$ and $a\in \unitr{A}$ 
\[
\suss{i}{a}-\suss{i}{-a^{-1}}=\suss{i}{-1}.
\]
\end{cor}
\begin{proof}
For $a\in \unitr{A}$, since $\an{-1}\suss{i}{a^{-1}}=\suss{i}{a}$ we have  
 \[
\suss{i}{-a^{-1}}=\suss{i}{-1\cdot a^{-1}}=\an{-1}\suss{i}{a^{-1}}+\suss{i}{-1}=\suss{i}{a}+\suss{i}{-1}.
\]
\end{proof}

\begin{lem}\label{lem:sus2}
 For $i=1,2$ and for all $x\in \unitr{A}$
\begin{enumerate}
\item $\lambda_1(\suss{i}{x})=-\pp{-1}\pf{x}=\pf{-x}\pf{x}\in \aug{A}^2$ 
\item $\lambda_2(\suss{i}{x})=(-x)\asymm x \in \asym{2}{\Z}{\unitr{A}}$.
\end{enumerate}
\end{lem}

\begin{proof}
\begin{enumerate}
\item For $x\in \wn{A}$, this is a straightforward calculation given in \cite[Lemma 3.3]{hut:rbl11}. For $x\in 
U_1$ it follows from the identity $\pf{xw}-\an{x}\pf{w}=\pf{x}$ in $\sgr{A}$.
\item Since $\unitr{A}$ acts trivially on $\asym{2}{\Z}{\unitr{A}}$, we have for any $x\in \wn{A}$ 
\begin{eqnarray*}
\lambda_2(\suss{1}{x})=\lambda_2(\suss{2}{x})&=&\lambda_2(\gpb{x})+\lambda_2(\gpb{x^{-1}})\\
&=& (1-x)\asymm x + (1-x^{-1})\asymm x^{-1}\\
&=& (1-x)\asymm x -\left(\frac{1-x}{-x}\right)\asymm x \\
&=& (-x)\asymm x.
\end{eqnarray*}

On the other hand, if $x\in U_1$, the result follows from the identity 
\[
(-xw)\asymm (xw)-(-w)\asymm w = (-x)\asymm x 
\]
in $\asym{2}{\Z}{\unitr{A}}$.
\end{enumerate}
\end{proof}

Let $\ks{i}{A}$ denote the $\sgr{A}$-submodule of $\rpb{A}$ generated by the set 
$\{ \suss{i}{x}\ |\ x\in \unitr{A}\}$. 

\begin{lem} \label{lem:kf}
 Then for $i\in \{ 1,2\}$
\[
\lambda_1\left( \ks{i}{A}\right) = \pp{-1}(\aug{A})\subset \aug{A}^2
\]
and $\ker{\lambda_1\res{\ks{i}{A}}}$ is annihilated by $4$.
\end{lem}
\begin{proof}
The first statement follows from Lemma \ref{lem:sus2}

For the second, the proof of Lemma 3.3 in \cite{hut:rbl11} applies without change.
\end{proof}

Let 
\[
\qrpb{A}:=\rpb{A}/\ks{1}{A} 
%\mbox{ and } \qasym{2}{\Z}{\unitr{A}}=\asym{2}{\Z}{\unitr{A}}/\an{ (-x)\asymm -x| x\in 
%\unitr{A}}.
\] 

By Lemma \ref{lem:sus2}, 
\[
\bw(\ks{1}{A})= \pp{-1}\aug{A}\oplus T(A)\subset \aug{A}^2\oplus \asym{2}{\Z}{\unitr{A}}
\]
where $T(A)$ is the subgroup of $\asym{2}{\Z}{\unitr{A}}$ generated by the elements 
$(-x)\asymm x$. Note that $T(A)$ is annihilated by $2$.

 We denote 
$\asym{2}{\Z}{\unitr{A}}/T(A)$ by  $\qasym{2}{\Z}{\unitr{A}}$.
Then the  $\sgr{A}$-homomorphism
\[
\rbw= (\tilde{\lambda}_1,\tilde{\lambda}_2):
\qrpb{A}\to \frac{\aug{A}^2}{\pp{-1}\aug{A}}\oplus \qasym{2}{\Z}{\unitr{A}}
\]
is well-defined and we set $\qrpbker{A}:=\ker{\tilde{\lambda}_1}$ and $\qrbl{A}:=\ker{\rbw}$.

\begin{cor}\label{cor:qrbl}
The natural maps $\rpbker{A}\to \qrpbker{A}$ and 
$\rbl{A}\to\qrbl{A}$ are surjective with kernel annihilated by $4$. In particular,
\[
\zhalf{\rpbker{A}}=\zhalf{\qrpbker{A}}\quad \mbox{ and }\quad \zhalf{\rbl{A}}=\zhalf{\qrbl{A}}.
\]
\end{cor}

\begin{proof}
From the definitions,  
\[
\qrpbker{A}\cong \frac{\rpbker{A}}{\rpbker{A}\cap \ks{1}{A}}\quad \mbox{ and }\quad 
\qrbl{A}\cong \frac{\rbl{A}}{\rbl{A}\cap \ks{1}{A}}
\]
and the modules $\rpbker{A}\cap \ks{1}{A}$ and $\rbl{A}\cap \ks{1}{A}$ are both contained in
$\ker{ \lambda_1|_{\ks{1}{A}}}$ which is annihilated by $4$. 
\end{proof}

\subsection{The constant $\cconst{A}$}

In \cite{sus:bloch}, Suslin shows that, for an infinite field $F$,
 the elements $\gpb{x}+\gpb{1-x}\in \bl{F}\subset \pb{F}$ are independent 
of $x$ and that the resulting constant, $\pbconst{F}$, has order dividing $6$. Furthermore Suslin shows 
that $\pbconst{\R}$ has exact order $6$.  

In \cite[Lemma 3.5]{hut:rbl11} it is shown that the elements 
\[
C(x)=\gpb{x}+\an{-1}\gpb{1-x}+\pf{1-x}\suss{1}{x}\in\rbl{F}
\]
are constant (for a field with at least $4$ elements) and have order dividing $6$. In fact the proof given there 
extends without alteration to the case of local rings:

\begin{lem}\label{lem:constbl}
Let $A$ be a local ring whose residue field has at least $4$ elements. For all $a, b \in \wn{A}$ 
\[
C(a):=\gpb{a}+\an{-1}\gpb{1-a}+\pf{1-a}\suss{1}{a}\in \rbl{A}\mbox{ and } C(a)=C(b).
\]
\end{lem}  

We denote this constant  by $\bconst{A}$ and we set $\cconst{A}:=2\bconst{A}$. Of course, $\bconst{A}$ 
maps to $\pbconst{A}$ under the natural map $\rpb{A}\to\pb{A}$. Similarly, we denote  the image 
of $\cconst{A}$ in $\pb{A}$ by $\pcconst{A}$. Thus, of course, $\pcconst{A}=2\pbconst{A}$ in $\pb{A}$. In fact, these 
elements lie in $\bl{A}$ by Lemma \ref{lem:constbl}.

Let $\Phi(X)$ denote the polynomial $X^2-X+1\in A[X]$.

\begin{lem}\label{lem:cconst}
Let $A$ be a local ring whose residue field has at least $4$ elements. 
\begin{enumerate}
\item $3\bconst{A}=\suss{1}{-1}$ and $6\bconst{A}=0$.
\item If $\Phi(X)$ has a root in $A$ then $\cconst{A}=0$ and $\bconst{A}=\suss{1}{-1}$. 
\end{enumerate}
\end{lem}
\begin{proof}
\begin{enumerate}
\item Let $a\in \wn{A}$. Then 
\begin{eqnarray*}
3\bconst{A}&=& C(a)+C(1-a^{-1})+C\left( \frac{1}{1-a}\right)\\
&=& \gpb{a}+\an{-1}\gpb{1-a}+\pf{\frac{1}{1-a}}\suss{1}{a}\\
&+& \gpb{1-a^{-1}}+\an{-1}\gpb{a^{-1}}+\pf{a}\suss{1}{1-a^{-1}}\\
&+& \gpb{\frac{1}{1-a}}+\an{-1}\gpb{\frac{1}{1-a^{-1}}}+\pf{1-a^{-1}}\suss{1}{\frac{1}{1-a}}\\
&=& \suss{1}{a}+\suss{1}{1-a^{-1}}+\suss{1}{\frac{1}{1-a}}\\
&+& \pf{\frac{1}{1-a}}\suss{1}{a}+\pf{a}\suss{1}{1-a^{-1}}+\pf{1-a^{-1}}\suss{1}{\frac{1}{1-a}}\\
&=& \an{\frac{1}{1-a}}\suss{1}{a}+\an{a}\suss{1}{1-a^{-1}}+\an{1-a^{-1}}\suss{1}{\frac{1}{1-a}}\\
&=& \suss{1}{\frac{1}{a^{-1}-1}}-\suss{1}{\frac{1}{1-a}}+\suss{1}{a-1}-\suss{1}{a}+\suss{1}{-\frac{1}{a}}-
\suss{1}{1-a^{-1}}\\
&=& 3\suss{1}{-1}=\suss{1}{-1}
\end{eqnarray*}
since for any $b$, we have 
\[
\suss{1}{b}-\suss{1}{-\frac{1}{b}}= \suss{1}{-1}
\]
by Corollary \ref{cor:psi-1}. 
\item  Let $a\in A$ with $\Phi(a)=0$. Then $1-a=a^{-1}$ and $a\in \wn{A}$. Furthermore, from $0=(a+1)\Phi(a)$ it follows 
that $a^3=-1$ and hence $a=-b^2$ where $b=a^2$. Thus $\suss{1}{a}=\suss{1}{-1}+\suss{1}{b^2}$ and hence $2\suss{1}{a}=0$. 
But 
\[
\bconst{A}=\gpb{a}+\an{-1}\gpb{1-a}-\pf{1-a}\suss{1}{(1-a)^{-1}}=\suss{1}{a}-\pf{a}\suss{1}{a}=-\an{a}\suss{1}{a}
\]
and hence $0=2\bconst{A}=\cconst{A}$. Thus also $\bconst{A}=3\bconst{A}=\suss{1}{-1}$.
\end{enumerate}
\end{proof}
%Furthermore, we let $\bconst{\F{2}}$ denote the distinguished generator 
%of $\rbl{\F{2}}$ of order $3$ and we set $\bconst{\F{3}}:= \suss{1}{-1}=(1+\an{-1})\gpb{-1}$.

From \cite{hut:rbl11} we also have the following identification of the element $\cconst{A}$:
\begin{lem}
For any $a\in \wn{A}$ we have 
\[
\cconst{A}=\gpb{a}+\an{-1}\gpb{\frac{1}{1-a^{-1}}}-\suss{1}{\frac{1}{1-a}}.
\]
\end{lem}

\begin{prop}[\cite{hut:arxivh3sl2dv}]\label{prop:cconstf} 
For a field $F$ we have $\cconst{F}=0$ if and only if $\Phi(X)$ has
a root in $F$. 
\end{prop}
\begin{cor}\label{cor:pcconst}
For any field $F$, $\cconst{F}=0$ if and only if $\pcconst{F}=0$, and the map 
$\rbl{F}\to\bl{F}$ induces an isomorphism $\Z\cdot\cconst{F}\cong\Z\cdot \pcconst{F}$.
\end{cor}

\begin{cor} \label{cor:pcconsta}
Let $A$ be a local ring with residue field $k$. Suppose either that $U_{1,A}=U_{1,A}^2$ and that 
$\mathrm{char}(k)\not=2,3$ or that $\mathrm{char}(k)=2$ and $A$ is henselian. 
Then the functorial map $\rbl{A}\to \rbl{k}$ induces an isomorphism 
$\Z\cdot \cconst{A}\cong \Z\cdot \cconst{k}$. 
\end{cor}
\begin{proof}
Clearly $\cconst{A}=0\imp \cconst{k}=0$. 

Conversely, if $\cconst{k}=0$ then $\Phi(X)$ has a root in $k$ . The hypotheses then imply 
that  $\Phi(X)$ has a root in $A$ and thus $\cconst{A}=0$ also.
\end{proof}

\begin{rem} When $\mathrm{char}(k)=3$, the result may fail.  
For example, $\Phi(X)$ has no root in $\Q_3$, so that 
$\cconst{\Q_3}\not= 0$, and hence $\cconst{\Z_3}\not=0$ also. But $\cconst{\F{3}}=0$. 
\end{rem}

\begin{thm}\label{thm:da}
Let $A$ be a local ring whose residue field has at least $10$ elements. Then 
\begin{enumerate}
\item For all $x\in \unitr{A}$, $\pf{x}\cconst{A}=\suss{1}{x}-\suss{2}{x}$.
\item  $\pf{x}\cconst{A}=0$ if $x\in \unitr{A}$ is of the form $\pm\Phi(a)u^2$ for some $a,u\in \unitr{A}$. 
\end{enumerate}
\end{thm}
\begin{proof} The proof when $A$ is a field is \cite[Theorem 3.12]{hut:rbl11}. The details of 
the extension to local rings is give in Appendix \ref{sec:app} below. 
\end{proof}

\subsection{The module $\cconstmod{A}$}

We let $\cconstmod{A}$ denote the cyclic $\sgr{A}$-submodule of $\rbl{A}$ generated by $\cconst{A}$. Of course, 
since $3\cconst{A}=0$ always, in fact $\cconstmod{F}$ is a module over the group ring $\F{3}[\sq{A}]$.

For a local ring $A$ with residue field $k$, 
let $\twn{A}:=\wn{A}\setminus\{ a\in A|\ \Phi(\bar{a})=0 \mbox{ in } k\}$.
Let $\tnorm{A}$ be the subgroup of $\unitr{A}$ generated by elements of the form 
$\pm\Phi(x)u^2$, $x\in \twn{A},u\in \unitr{A}$. By Theorem \ref{thm:da} (2), we have $\an{x}\cconst{A}=\cconst{A}$ if 
$x\in \tnorm{A}$. Let $\rsq{A}=\unitr{A}/\tnorm{F}$. Thus the action of 
$\sq{A}$ on $\cconst{A}$ factors through the quotient $\rsq{A}$, and hence $\cconstmod{A}$ is a cyclic module 
over the ring $\rsgr{A}:=\F{3}[\rsq{A}]$.

\begin{lem}\label{lem:da2df}
Let $A$ be a local ring with residue field $k$. Then the natural map $\cconstmod{A}\to\cconstmod{k}$ is surjective.

Furthermore, if $U_{1,A}=U_{1,A}^2$ and if $\cconstmod{k}$ is free (of rank $1$) as a $\rsgr{k}$-module, then  
$\cconstmod{A}\cong\cconstmod{k}$. 
\end{lem}

\begin{proof}
Since $\sq{A}$ maps onto $\sq{F}$ and hence $\sgr{A}$ maps onto $\sgr{F}$ the first statement is clear.

For the second statement, note that the conditions ensure that $\sq{A}\cong\sq{k}$ and under this isomorphism,
$\tnorm{A}$ corresponds to $\tnorm{k}$. Thus $\rsgr{A}\cong\rsgr{k}$ and the composite map
\[
\xymatrix{
\rsgr{A}\ar@{->>}[r]
&
\cconstmod{A}\ar[r]
&
\cconstmod{k}
}
\]  
is an isomorphism. 
\end{proof}

From \cite[Corollary 7.11, Proposition 7.12] {hut:arxivh3sl2dv} we itemize some conditions under which 
$\cconstmod{k}$ is a free $\rsgr{k}$-module.

\begin{prop}\label{prop:dkfree}
Let $k$ be a field in which $X^2-X+1$ has no roots. Then $\cconstmod{k}$ is a free $\rsgr{k}$-module 
of rank $1$ in any of the following cases
\begin{enumerate}
\item $k$ is finite.
\item $k$ is real closed.
\item There is a discrete valuation on $k$ with residue field $\bar{k}$ satisfying one of the 
following conditions:
\begin{enumerate}
\item $U_1=U_1^2$ and $X^2-X+1$ has no root in $\bar{k}$ and $\cconstmod{\bar{k}}$ is a free 
$\rsgr{\bar{k}}$-module
\item $k$ is complete with respect to $v$ and $\bar{k}$ is finite of characteristic not equal to 
$3$
\item $k$ is complete with respect to $v$ and $\bar{k}$ is finite of characteristic  
$3$ and $[k:\Q_3]$ is odd. 
\end{enumerate}
\end{enumerate}
\end{prop}

\begin{rem}
\begin{enumerate}
\item If $k$ is complete with respect to the valuation $v$ and if $\chara{\bar{k}}\not=2$ then 
$U_1=U_1^2$. However, $\Q_2$ is complete with respect to the associated discrete valuation but 
$U_1\not= U_1^2$ in this case.
\item $X^2-X+1$ has a (repeated) root in a $\F{3}$, but does not have a root in $\Q_3$.
Thus $\cconstmod{\F{3}}=0$ but $\cconstmod{\Q_3}\not= 0$.  However, 
$\aug{\Q_3}\cconstmod{\Q_3}=0$ by Proposition \ref{prop:dkfree} (3)(c), 
since $\rsgr{\Q_3}=\F{3}$ by \cite[Lemma 6.2]{hut:rbl11}. 
\end{enumerate}
\end{rem}

\subsection{Reduced  scissors  congruence groups }
We introduce some quotient groups of scissors congruence
 groups which will be required in our computations below. 

First recall that 
\[
\qrpb{A}=\frac{\rpb{A}}{\ks{1}{A}},\quad 
\qrpbker{A}=\ker{\tilde{\lambda}_1}\quad \mbox{ and } \quad \qrbl{A}=\ker{\rbw}\subset \qrpbker{A}.
\] 
Furthermore, by Corollary \ref{cor:qrbl}),
 the maps $\rpbker{A}\to\qrpbker{A}$ and $\rbl{A}\to\qrbl{A}$ are surjective with kernel 
annihilated by $4$. 

Since 
$\cconstmod{A}$ is annihilated by $3$, it follows that the composite 
\[
\cconstmod{A}\to\rbl{A}\to\qrbl{A}
\]  
is injective and we will identify $\cconstmod{A}$ with its image in $\qrbl{A}$.

We further define
\[
\rrpb{A}:=\frac{\qrpb{A}}{\cconstmod{A}}=\frac{\rpb{A}}{\ks{1}{A}+\cconstmod{A}},\quad 
\rrpbker{A}:=\frac{\qrpbker{A}}{\cconstmod{A}}
\quad 
\mbox{ and }\quad \rrbl{A}:=\frac{\qrbl{A}}{\cconstmod{A}}.
\]

Corollary \ref{cor:qrbl} implies:
\begin{lem}\label{lem:rbl}
For any field or local ring $A$ there are  short exact sequences
\[
0\to\cconstmod{A}\to\zhalf{\rpbker{A}}\to\zhalf{\rrpbker{A}}\to 0.
\]
and 
\[
0\to\cconstmod{A}\to\zhalf{\rbl{A}}\to\zhalf{\rrbl{A}}\to 0.
\]
\end{lem} 

Observe also that $\ks{2}{A}\subset\ks{1}{A}+\cconstmod{A}$ by Theorem \ref{thm:da} (2). It follows that for 
$i=1,2$ and all $x\in\unitr{A}$ we have 
$\suss{i}{x}=0$ in $\rrpb{A}$. 

\begin{lem}\label{lem:pb-} Let $A$ be a local ring whose residue field has at least $10$ elements.
\begin{enumerate}
\item For all $x\in \wn{A}$, $\gpb{x^{-1}}=-\an{-1}\gpb{x}=-\an{x}\gpb{x}$ in $\rrpb{A}$.
\item For all $x\in \wn{A}$, $\pf{y}\gpb{x}=0$ in $\rrpb{A}$ whenever 
$y\equiv -x\pmod{(\unitr{A})^2}$.
\item For all $x\in \wn{A}$, $\gpb{1-x}=\an{-1}\gpb{x}=\gpb{x^{-1}}$ in $\rrpb{A}$.
\end{enumerate}
\end{lem}
\begin{proof}
\begin{enumerate}
\item The first equality follows from $\suss{1}{x}=0$, the second from $\suss{2}{x}=0$.
\item From $\an{x}\gpb{x}=\an{-1}\gpb{x}$ it follows that $\an{y}\gpb{x}=\an{-x}\gpb{x}=\gpb{x}$.
\item Since $2\bconst{A}=\cconst{A}$ and $3\bconst{A}=\suss{1}{-1}\in\ks{1}{A}$, 
it follows that $\bconst{A}=0$ in $\rrpb{A}$, and thus we have (from the definition of 
$\bconst{A}$) that $0=\gpb{x}+\an{-1}\gpb{1-x}$ in $\rrpb{A}$.
\end{enumerate}
\end{proof}

\section{Characters and $\sgr{A}$-modules}\label{sec:char}
In this section we review some results from \cite{hut:arxivh3sl2dv}.

\subsection{The character-theoretic local-global principle}
Let $G$ be a multiplicative abelian group satisying $g^2=1$ for all $g\in G$. 
Let $\mathcal{R}$ denote
 the group ring $\Z[G]$. Let $\dual{G}$ be the group of characters of $G$, i.e. the group of 
homomorphisms $\chi:G\to \mu_2$. 

For any $\mathcal{R}$-module $M$ and any $\chi\in\dual{G}$ we let $\dwn{M}{\chi}$ denote the the 
quotient of $M$ by the submodule generated by the elements $(\chi(g)-g)m$, $g\in G$, $m\in M$. 
Thus $g\cdot x= \chi(g)x$ for all $g\in G$ and $x\in \dwn{M}{\chi}$. For any $m\in M$, we 
will let $\dwn{m}{\chi}$ denote the image of $m$ in $\dwn{M}{\chi}$.

We let $\chi_0$ denote the trivial character on $G$. Thus, for any $\mathcal{R}$-module 
$M$,  $\dwn{M}{\chi_0}$ is equal to the module $M_G$ of coinvariants. 

If $f:M\to N$ is a homomorphism of $\mathcal{R}$-modules we let $\dwn{f}{\chi}$ denote the induced 
homomorphism $\dwn{M}{\chi}\to\dwn{N}{\chi}$.  

The following observation is proved in \cite[Corollary 3.7]{hut:arxivh3sl2dv}:

\begin{lem}\label{lem:exact}
For any $\chi\in\dual{G}$, the functor $M\to \dwn{M}{\chi}$ is an exact functor on the category 
of $\zhalf{\mathcal{R}}$-modules.  
\end{lem}

We recall the following character-theoretic local global principle 
(see \cite[Proposition 3.10]{hut:arxivh3sl2dv}):

\begin{prop}\label{prop:lg}
Let $f:M\to N$ be a homomorphism of $\zhalf{\mathcal{R}}$-modules. Then $f$ is injective (resp. surjective)
if and only if $\dwn{f}{\chi}$ is injective (resp. surjective) for all $\chi\in \dual{G}$. 
\end{prop}

Of course we will apply this in the case $G=\sq{A}$ and hence $\mathcal{R}=\sgr{A}$ for a local ring 
$A$.

\begin{lem} \label{lem:rpbchi} Let $A$ be a local ring whose residue field has at least $10$ 
elements.
Let $\chi\in\dual{\sq{A}}$. 

If $a\in \wn{A}$ satisfies $\chi(-a)=-1$ then 
$2\cdot \dwn{\gpb{a}}{\chi}=0$ in $\dwn{\rrpb{A}}{\chi}$.
\end{lem}

\begin{proof}
By definition, $(\an{-a}+1)x=0$ for all $x\in \dwn{\rrpb{A}}{\chi}$. In particular, 
$(\an{-a}+1)\dwn{\gpb{a}}{\chi}=0$
However,
$(\an{-a}-1)\gpb{a}=0$ in $\rrpb{A}$ by Lemma \ref{lem:pb-} (2).  
\end{proof}
\subsection{The action of $\an{-1}$ on $\rpbker{A}$}

\begin{prop}\label{prop:-1} Let $A$ be a local ring whose residue field has at least 
$10$ elements. Then $\an{-1}\in \sgr{F}$  acts trivially on $\zhalf{\rrpbker{A}}$ and in fact 
$\zhalf{\rpbker{A}}=\ep{-1}\zhalf{\rpb{A}}$.  
\end{prop}

\begin{proof} The proof given in \cite[Theorem 4.3, Corollary 4.4]{hut:arxivh3sl2dv} 
in the case of a field $F$  applies here, 
simply by replacing $F^\times\setminus\{ 1\}$ with $\wn{A}$. 
 \end{proof}

Since $\an{-1}$ acts trivially on $\cconst{A}$ and on $\mathrm{tor}{(\mu_A,\mu_A)}$, we deduce as in 
\cite{hut:arxivh3sl2dv}:

\begin{cor}\label{cor:-1}
 Let $A$ be a local ring whose residue field has at least 
$10$ elements.
 Then $\an{-1}\in \sgr{A}$  acts trivially on 
$\zhalf{\rpbker{A}}$ and on $\ho{3}{\spl{2}{A}}{\zhalf{\Z}}$.
\end{cor}

\begin{rem} As in \cite{hut:arxivh3sl2dv}, it follows that 
$\zhalf{\rpbker{A}}= \ep{-}\zhalf{\rpb{A}}$ is a quotient of $\zhalf{\rpb{A}}$ and hence admits a 
simple explicit presentation as a $\sgr{A}$-module.
\end{rem}

\section{ The map from $\rpbker{A}$ to $\rpbker{k}$}\label{sec:val}

In this section $A$ will be a commutative local ring with residue field $k$.

\subsection{The module $\kv{A}$} 

We let 
\[
\kv{A}:= \modgen{\gpb{au}-\gpb{a}\in \rrpb{A}}{a\in \wn{A},u\in U_1}{\sgr{A}}.
\]

%The following is an immediate consequence of the definition:

\begin{lem}\label{lem:kv} There is a short exact sequence of $\sgr{A}$-modules 
\[
0\to \kv{A}\to \rrpb{A}\to \rrpb{k}\to 0.
\]
\end{lem}

\begin{proof} The functorial map $\rrpb{A}\to \rrpb{k}$ is clearly surjective and $\kv{A}$ is 
contained in its kernel.
 
Thus if we let $Q(A):=\rrpb{A}/\kv{A}$, there is an induced surjective homomorphism 
of $\sgr{A}$-modules $Q(A)\to \rrpb{k}$. We must show that this is an isomorphism. 

I \emph{claim} that $\an{u}$ acts trivially on $Q(A)$ for any $u\in U_1$: Let $a\in \wn{A}$. 
Then $\gpb{a}=\gpb{au}$ in $Q(A)$ for any $u\in U_1$.
 For all $b\in \wn{A}$,  $\an{-b}\gpb{b}=\gpb{b}$ in $Q(A)$ by Lemma \ref{lem:pb-} (2).
Thus, for any $u\in U_1$ we have
\[
\an{u}\gpb{a}=\an{-au}\gpb{a}=\an{-au}\gpb{au}=\gpb{au}=\gpb{a}
\] 
in $Q(A)$, proving the \emph{claim}.

It follows that the action of $\sgr{A}$ on $Q(A)$ descends to an action of 
$\Z[A^\times/U_1\cdot(A^\times)^2]=\sgr{k}$. 
This allows us to construct a well-defined inverse map 
\[
\rrpb{k}\to Q(A)=\rrpb{A}/\kv{A},\quad \gpb{\bar{a}}\mapsto \gpb{a}+\kv{A}.
\] 
\end{proof}
\subsection{The  map $\rpbker{A}\to\rpbker{k}$}

\begin{thm}\label{thm:main}
Let $A$ be a local ring with sufficiently large residue field $k$. Suppose that $A$ satisfies 
the condition $U_1=U_1^2$. 

Then the  functorial map $\rpb{A}\to\rpb{k}$ induces an isomorphism 
\[
\aug{A}\zhalf{\rrpbker{A}}\cong\aug{k}\zhalf{\rrpbker{k}}.
\]
\end{thm}

\begin{proof}
By Proposition \ref{prop:-1} and the character-theoretic local-global principle, it is 
enough to prove that $\dwn{\zhalf{\rrpb{A}}}{\chi}\cong \dwn{\zhalf{\rrpb{k}}}{\chi}$ for 
all characters $\chi$ satisfying $\chi(-1)=1$ and $\chi\not=\chi_0$.  

By Lemmas \ref{lem:exact} and 
\ref{lem:kv} we must prove that $\dwn{\zhalf{\kv{A}}}{\chi}=0$ for all such characters 
$\chi$; i.e. we must prove that $\dwn{\gpb{a}}{\chi}=\dwn{\gpb{au}}{\chi}$ for all 
$a\in \wn{A}$, $u\in U_1$ and $\chi$ satisfying $\chi(-1)=1$, $\chi\not=\chi_0$. 

Let $\chi\not= \chi_0$ be a character satisfying $\chi(-1)=1$. Let 
$K=K(\chi):=\{ a\in \unitr{A}\ |\ \chi(a)=1\}$.    

By Lemma \ref{lem:rpbchi}, if $a\in\wn{A}\setminus K$ then $\dwn{\gpb{a}}{\chi}=0$ in 
$\dwn{\zhalf{\rrpb{A}}}{\chi}$. Since $\gpb{1-a}= \an{-1}\gpb{a}$ in $\rrpb{A}$, it also 
follows that if $1-a\not\in K$ then $\dwn{\gpb{a}}{\chi}=0$.

Let $S_1:=\{a\in \wn{A}\ |\ a,1-a\in K\}$ and let $S_2:=\{ a\in \wn{A}\ |\ a\in K, 1-a\not\in K\}$. 
Thus $\dwn{\gpb{a}}{\chi}=0$ if $a\not\in S_1$.

Observe that if $a\in \wn{A}\cap K$ then 
\[
\chi(1-a^{-1})=\chi(-a^{-1}(1-a))=\chi(-1)\chi(a)\chi(1-a)=\chi(1-a).
\]
Thus $a\in S_i$ if and only if $a^{-1}\in S_i$, for $i=1,2$.

Furthermore, $U_1\subset K$ since $U_1=U_1^2$ and for any $a\in \wn{A}$ we have 
$1-a\equiv 1-au\pmod{U_1} \imp 1-a\equiv 1-au\pmod{(\unitr{A})^2}$ and thus 
$\chi(1-a)=\chi(1-au)$.  

It follows that for all $a\in \wn{A}$ and $u\in U_1$ we have 
$a\in S_i$ if and only if $au\in S_i$ for $i=1,2$. 

Note also that $S_2\not=\emptyset$: Let $a\not\in K$. If $1-a\in K$ then $1-a\in S_2$. Otherwise 
$1-a\not\in K$ and hence $\chi(1-a^{-1})=\chi(-1)\chi(a)\chi(1-a)=1$ so that $1-a^{-1}\in S_2$. 

Now let $a\in S_1$, $b\in S_2$. Note that it follows that $ab\in \wn{A}$  since 
otherwise $u:=ab\in U_1\imp b=a^{-1}u\in S_1$. 

We \emph{Claim} that $\dwn{\gpb{a}}{\chi}=\dwn{\gpb{ab}}{\chi}$: We have 
\[
0=\dwn{\left( S_{a,ab}\right)}{\chi}=
\dwn{\gpb{a}}{\chi}-\dwn{\gpb{ab}}{\chi}+\dwn{\gpb{b}}{\chi}-
\dwn{\gpb{b\cdot\frac{1-a}{1-ab}}}{\chi}+\dwn{\gpb{\frac{1-a}{1-ab}}}{\chi}.
\]

Note that $\dwn{\gpb{b}}{\chi}=0$ since $b\not\in S_1$. 

We will show that the last two 
terms also vanish.
We consider two cases:
\begin{enumerate}
\item[Case (i): ] $ab\in S_2$

In this case, $1-ab\not\in K \imp (1-a)/(1-ab)\not\in K$ and $b\cdot(1-a)/(1-ab)\not\in K$. Hence 
the last two terms are $0$.

\item[Case (ii):] $ab\not\in S_2$

Then $1-ab\in K$. Thus $s:=(1-a)/(1-ab)\in K$. However
\[
1-s=1-\frac{1-a}{1-ab}=a\cdot \frac{1-b}{1-ab}\not\in K
\]
since $a\in K$, $1-b\not\in K$. it follows that $\dwn{\gpb{s}}{\chi}=0$. 

Similarly, replacing $a, b$ by $a^{-1},b^{-1}$ it follows that $\dwn{\gpb{t}}{\chi}=0$ where 
\[
t:= \frac{b(1-a)}{1-ab}= \frac{1-a^{-1}}{1-(ab)^{-1}}. 
\]
\end{enumerate}
Thus the \emph{Claim} is proven. 

Now let $a\in \wn{A}$, $u\in U_1$. Fix $b\in S_2$.

If $a\not\in S_1$ then $au\not\in S_1$ and hence 
\[
\dwn{\gpb{a}}{\chi}=0=\dwn{\gpb{au}}{\chi}.
\]

Otherwise, $a\in S_1$. Hence $au\in S_1$ and $u^{-1}b\in S_2$. By the \emph{Claim} we have 
\[
\dwn{\gpb{a}}{\chi}=\dwn{\gpb{ab}}{\chi}
\]
 and also 
\[
\dwn{\gpb{au}}{\chi} = \dwn{\gpb{(au)(u^{-1}b)}}{\chi} = \dwn{\gpb{ab}}{\chi}.
\]
The theorem is thus proven. 
\end{proof}

\begin{rem} The theorem does not hold for general local rings, however.

For example, if the residue field $k$ is finite, then $\aug{k}\rpbker{k}=0$. Thus if the theorem 
holds for a local ring $A$ with finite residue field then it implies that 
$\aug{A}\zhalf{\rrpbker{A}}=0$. 

Let $p$ be a prime number and consider $A=\Z_{\il{p}}= \{ a/b\in \Q\ |\ p\not| b\ \}$ with residue 
field $\F{p}$. In \cite{hut:rbl11} it is shown that there is a natural surjective homomorphism
\[
\aug{\Q}\zhalf{\rpbker{\Q}}\to \oplus_{q\ \mbox{\tiny prime}} \zhalf{\pb{\F{q}}}.
\] 
It is straightfoward to verify that the resulting  composite map
 \[
\aug{A}\zhalf{\rpbker{A}}\to \aug{\Q}\zhalf{\rpbker{\Q}}
\to \oplus_{q\not= p} \zhalf{\pb{\F{q}}}
\]
is surjective and thus $\aug{A}\zhalf{\rpbker{A}}\not= 0$ and also 
$\aug{A}\zhalf{\rrpbker{A}}\not=0$.

This example suggests that some completeness or Henselian condition, such as the condition
 $U_1=U_1^2$, is necessary for the theorem to hold. 
\end{rem}

\subsection{Application to the third homology of $\mathrm{SL}_2$ of local domains}

For a local ring $A$ with residue field $k$ we let 
\[
\mathcal{D}(A/k):= \ker{\aug{A}\cconstmod{A}\to \aug{k}\cconstmod{k}}.
\]
Thus $\mathcal{D}(A/k)$ is an $\sgr{A}$-submodule of $\cconstmod{A}$. In particular, 
it is annihilated 
by $3$. 

\begin{rem}\label{rem:dfin}
Note that, by Lemma \ref{lem:da2df} and Proposition \ref{prop:dkfree} (1), if $A$ is a local integral
domain satisfying $U_1=U_1^2$ and with finite residue field $k$ then $\mathcal{D}(A/k)=0$.

In fact, when $k$ is a finite field then $\sgr{k}$ acts trivially on $\rbl{k}=\bl{k}$ and hence 
$\aug{k}\cconstmod{k}=0$.  Thus, under the stated  conditions, $\aug{A}\cconstmod{A}=0$ also.  
\end{rem}

\begin{prop}\label{prop:h3afin}
Let $A$ be a local integral domain with sufficiently large finite 
residue field $k$. Suppose that $A$ satisfies 
the condition $U_1=U_1^2$. 

Then $A^\times$ acts trivially on $\ho{3}{\spl{2}{A}}{\zhalf{\Z}}$.

Equivalently,
 the natural map $\ho{3}{\spl{2}{A}}{\zhalf{\Z}}\to\zhalf{\kind{A}}$ 
is an isomorphism. 
\end{prop}

\begin{proof} By Corollaries \ref{cor:bwfin} and \ref{cor:rpbkerA} it is enough to prove that 
$\aug{A}\zhalf{\rpbker{A}}=0$. 
  
There is a commutative diagram of $\sgr{A}$-modules with exact rows
\[
\xymatrix{
0\ar[r]
&\aug{A}\cconstmod{A}
\ar[r]\ar@{>>}[d]
&\aug{A}\zhalf{\rpbker{A}}\ar[r]\ar@{>>}[d]
&\aug{A}\zhalf{\rrpbker{A}}\ar[r]\ar[d]^-{\cong}
&0\\
0\ar[r]
&\aug{k}\cconstmod{k}
\ar[r]
&\aug{k}\zhalf{\rpbker{k}}\ar[r]
&\aug{k}\zhalf{\rrpbker{k}}\ar[r]
&0.\\}
\]

The right-hand vertical arrow is an isomorphism by Theorem \ref{thm:main}. Thus we can extract the 
natural short exact sequence
\[
0\to \mathcal{D}(A/k)\to \aug{A}\zhalf{\rpbker{A}}\to \aug{k}\zhalf{\rpbker{k}}\to 0.
\]

By remark \ref{rem:dfin}, $\mathcal{D}(A/k)=0$ and hence the middle vertical arrow is an isomorphism.
 
Furthermore, since $k$ is finite, $\rbl{k}=\bl{k}$ and hence 
$\aug{k}\rbl{k}=0$ (\cite[Lemma 7.1]{hut:cplx13}). It follows that 
\[
\aug{A}\zhalf{\rbl{A}}= \aug{A}\zhalf{\rpbker{A}}\cong
\aug{k}\zhalf{\rpbker{k}}= \aug{k}\zhalf{\rbl{k}}=0. 
\] 
\end{proof}

\begin{cor}\label{cor:main}
Let $A$ be a local integral domain with sufficiently large residue field $k$. Suppose that 
$U_1=U_1^2$. There is a short exact 
sequence of $\sgr{A}$-modules
\[
0\to \mathcal{D}(A/k)\to \hot{A}{\zhalf{\Z}}\to \hot{k}{\zhalf{\Z}}\to 0.
\] 
\end{cor}

\begin{proof}
%When $k$ is finite, all terms in this sequence are zero and there is nothing more to prove. 
%Thus we can assume that $k$ is infinite. 
As in the proof of Proposition \ref{prop:h3afin}, there is a short exact sequence
\[
0\to \mathcal{D}(A/k)\to \aug{A}\zhalf{\rpbker{A}}\to \aug{k}\zhalf{\rpbker{k}}\to 0.
\]

We have 
\[
\aug{A}\zhalf{\rpbker{A}}\cong \hot{A}{\zhalf{\Z}}\mbox{ and }
\aug{k}\zhalf{\rpbker{k}}\cong \hot{k}{\zhalf{\Z}}
\]
by Corollary \ref{cor:rpbkerA}.
\end{proof}

\section{The localization sequence for discrete valuation rings}\label{sec:local}

\subsection{Fields with discrete valuation}

We recall some of the main results of \cite{hut:arxivh3sl2dv}: 

Let $F$ be a field with discrete valuation
$v$ and residue field $k=k_v$. Let $\mathcal{O}_v$ denote the corresponding valuation ring. 
Let $U=U_v=\unitr{\mathcal{O}_v}$.

Fix a uniformizer $\pi\in\mathcal{O}_v$.
 Then exists a surjective homomorphism of $\Z[U_v/U_v^2]$-modules 
$\delta_\pi:\hoz{3}{\spl{2}{F}}{\zhalf{\Z}}\to \zhalf{\rpbker{k}}$, which we proceed to describe. 
(We also denote by $\delta_\pi$ 
the restriction of this homomorphism to $\hot{F}{\zhalf{\Z}}$.)

Given any $\sgr{k}$-module $M$, 
we define the induced $\sgr{F}$-module
\[
\indf{k}{F}{M}:=\sgr{F}\otimes_{\Z[U/U^2]}M.
\]

 Note that $\sq{F}=U/U_2\times C_\pi$ where 
$C_\pi$ denotes the multiplicative group of order $2$ generated by the square class $\an{\pi}$.
Thus  $\sgr{F}=\Z[U/U^2][C_\pi]$  and there is a decomposition of $\Z[U/U^2]$-modules
\[
\indf{k}{F}{M}\cong M\oplus \an{\pi}\cdot M\cong M\oplus M.
\]
We let $\rho_\pi:\indf{k}{F}{M}\to M$ denote the projection on the second factor. 

There is a natural surjective specialization or residue
homomorphism of $\sgr{F}$-modules 
\begin{eqnarray*}
S:=S_{v}:\rpb{F}&\to&\indf{k}{F}{\qrpb{k}}\\
\gpb{a}&\mapsto&\left\{
\begin{array}{rc}
1\otimes \gpb{\bar{a}},& v(a)=0\\
1\otimes\bconst{k},& v(a)>0\\
-(1\otimes\bconst{k}),& v(a)<0\\
\end{array}
\right.
\end{eqnarray*}
where $\qrpb{k}$ is a certain quotient of the $\sgr{k}$-module $\rpb{k}$. 

$S_v$ induces a well-defined homomorphism
\[
S_v:\rpbker{F}\to \indf{k}{F}{\qrpbker{k}}
\]
where $\qrpbker{k}$ is a quotient of $\rpbker{k}$ satisfying $\zhalf{\rpbker{k}}\cong\zhalf{\qrpbker{k}}$. 
(However, $S_v$ does not restrict to a homomorphism $\rbl{F}\to \indf{k}{F}{\qrbl{k}}$.)

Then we define $\delta_\pi$ to be the composite
\[
\xymatrix{
\ho{3}{\spl{2}{F}}{\zhalf{\Z}}\ar[r]
&\zhalf{\rbl{F}}\ar[r]
&\zhalf{\rpbker{F}}\ar^-{S_v}[r]
&\indf{k}{F}{\zhalf{\rpbker{k}}}\ar^-{\rho_\pi}[r]
&\zhalf{\rpbker{k}}.
}
\]

\begin{lem}\label{lem:dpiov}
The composite 
\[
\xymatrix{
\ho{3}{\spl{2}{\mathcal{O}_v}}{\zhalf{\Z}}\ar[r]
&\ho{3}{\spl{2}{F}}{\zhalf{\Z}}\ar^-{\delta_\pi}[r]
&\zhalf{\rpbker{k}}
}
\]
is the zero map. 
\end{lem}

\begin{proof}
By functoriality, the square of $\Z[U/U^2]$-module homomorphisms
\[
\xymatrix{
\ho{3}{\spl{2}{\mathcal{O}_v}}{\Z}\ar[r]\ar[d]
&\ho{3}{\spl{2}{F}}{\Z}\ar[d]\\
\rbl{\mathcal{O}_v}\ar[r]
&\rbl{F}
}
\]
commutes. Since $\delta_\pi$ factors through the right vertical map, it is sufficient to prove that 
the composite
\[
\xymatrix{
\zhalf{\rbl{\mathcal{O}_v}}\ar[r]
&\zhalf{\rbl{F}}\ar^-{S_v}[r]
&\indf{k}{F}{\zhalf{\qrpb{k}}}\ar^-{\rho_\pi}[r] 
&\zhalf{\qrpb{k}}
}
\]
is the zero map.

$\rbl{\mathcal{O}_v}$ is generated as a $\Z[U/U^2]$-module by elements $\gpb{a}$, $a\in U\setminus U_1$.
By definition, 
\[
\rho_\pi(S_v(\gpb{a}))=\rho_\pi(1\otimes \gpb{\bar{a}})=0
\]
for any such $a$. 
\end{proof}

Let $F$ be a field with discrete valuation $v:F^\times \to \Z$. We consider the module
\[
\cconstmod{F}(v):= \ker{\spec{v}: \aug{F}\cconstmod{F}\to \aug{F}(\indf{k}{F}{\cconstmod{k}})}.
\]
Observe that this module is annihilated by $3$.

\begin{prop}\label{prop:h3sl2dv}
Let $F$ be a field with discrete valuation
$v$. Let $\pi$ be a uniformizer. Suppose that there exists $n\geq 1$ such that $U_n\subset U_1^2$. 
Then there is a natural short exact sequence of $\Z[U_v/U_v^2]$-modules
\[
0\to \cconstmod{F}(v)\to \hot{F}{\zhalf{\Z}}\to \hot{k}{\zhalf{\Z}}\oplus \zhalf{\rpbker{k}}\to 0
\]
where $\delta_\pi$ is the composite of the map
$
\hot{F}{\zhalf{\Z}}\to \hot{k}{\zhalf{\Z}}\oplus \zhalf{\rpbker{k}}
$
with projection on the second factor. 
\end{prop}

\begin{proof} This is \cite[Proposition 7.1]{hut:arxivh3sl2dv}.
\end{proof}
\begin{cor}\label{cor:dv}
Let $F$ be a field with discrete valuation $v$ and residue field $k$. Suppose that 
$U_1= U_1^2$ and that $\cconstmod{F}(v)=0$. Let $\pi$ be a uniformizer.
 
Then there is a short exact sequence of $\sgr{k}$-modules
\[
\xymatrix{
0\ar[r]
& \hot{k}{\zhalf{\Z}}\ar[r]
&\hot{F}{\zhalf{\Z}}\ar[r]^-{\delta_\pi}
&\zhalf{\rpbker{k}}\ar[r]
&0.\\
}
\]
\end{cor}

\subsection{The localization theorem}
\begin{thm}\label{thm:loc}
Let $F$ be a field with discrete valuation $v$ and residue field $k$. Suppose that either 
$\mathrm{char}(F)=\mathrm{char}(k)$ or that $k$ is algebraic over $\F{p}$ for some prime $p$.
Suppose further that 
$U_1= U_1^2$ and that $\cconstmod{F}(v)=0=\mathcal{D}(\mathcal{O}_v/k)$. Let $\pi$ be a uniformizer.
 
Then there is a natural short exact sequence of $\sgr{\mathcal{O}_v}$-modules
\[
\xymatrix{
0\ar[r]
& \ho{3}{\spl{2}{\mathcal{O}_v}}{\zhalf{\Z}}\ar[r]
&\ho{3}{\spl{2}{F}}{\zhalf{\Z}}\ar[r]^-{\delta_\pi}
&\zhalf{\rpbker{k}}\ar[r]
&0.\\
}
\]
\end{thm}
\begin{proof}
We begin by observing that the triangle 
\[
\xymatrix{
\hot{\mathcal{O}_v}{\zhalf{\Z}}\ar[rr]^-{\cong}\ar[rd]
&&\hot{k}{\zhalf{\Z}}\ar[ld]\\
&\hot{F}{\zhalf{\Z}}&
}
\]
commutes (where the horizontal map is an isomorphism by Corollary \ref{cor:main}). This follows 
from the commutativity of the square
\[
\xymatrix{
\rpb{\mathcal{O}_v}\ar[r]\ar[d]
&\qrpb{k}\ar@{^{(}->}[d]\\
\rpb{F}\ar[r]^-{\spec{v}}
&\indf{k}{F}{\qrpb{k}}\\
},\quad
\xymatrix{
\an{u}\gpb{v}\ar[r]\ar[d]
&\an{\bar{u}}\gpb{\bar{v}}\ar[d]\\
\an{u}\gpb{v}\ar[r]
&1\otimes \an{\bar{u}}\gpb{\bar{v}}.\\}
\]  

Thus by Corollary \ref{cor:dv} the long vertical column in the diagram 
\[
\xymatrix{
&0\ar[d]&&&\\
0\ar[r]
&\hot{\mathcal{O}_v}{\zhalf{\Z}}\ar[r]\ar[d]
&\ho{3}{\spl{2}{\mathcal{O}_v}}{\zhalf{\Z}}\ar[r]\ar[d]
&\zhalf{\kind{\mathcal{O}_v}}\ar[r]\ar[d]^-{\cong}
&0\\
0\ar[r]
&\hot{F}{\zhalf{\Z}}\ar[r]\ar[d]^-{\delta_\pi}
&\ho{3}{\spl{2}{F}}{\zhalf{\Z}}\ar[r]
&\zhalf{\kind{F}}\ar[r]
&0\\
&\zhalf{\rpbker{k}}\ar[d]&&&\\
&0&&&\\
}
\]
is exact.   

Since the right-hand vertical arrow is an isomorphism by Theorem \ref{thm:ktheory}, the result follows 
from the snake lemma.

\end{proof}
\subsection{The condition $\cconstmod{F}(v)=0=\mathcal{D}(\mathcal{O}_v/k)$}
\begin{rem}
If $X^2-X+1$ has a root in $\mathcal{O}_v$, then $\cconst{\mathcal{O}_v}=\cconst{F}=\cconst{k}=0$
and therefore $\cconstmod{F}(v)=0=\mathcal{D}(\mathcal{O}_v/k)$. 
\end{rem}

\begin{prop}\label{prop:do}
Let $F$ be a field with discrete valuation $v$ and residue field $k$. Suppose that $U_1=U_1^2$ and 
that $\cconstmod{k}$ is free as a $\rsgr{k}$-module. Suppose also that $X^2-X+1$ has no root in
 $k$. Then
\[
\cconstmod{F}(v)=0=\mathcal{D}(\mathcal{O}_v/k).
\]
\end{prop}

\begin{proof}
By Lemma \ref{lem:da2df} the conditions ensure that the map $\cconstmod{\mathcal{O}_v}\to 
\cconstmod{k}$ is an isomorphism and hence that $\mathcal{D}(\mathcal{O}_v/k)=0$. 

On the other hand, $\cconstmod{F}(v)=0$ under these 
conditions by \cite[Corollary 7.11]{hut:arxivh3sl2dv}.
\end{proof}

\subsection{Some examples}
\begin{exa} Let $F$ be a field complete with respect to a discrete valuation $v$ with 
sufficiently large finite residue field $k$ of odd characteristic. 
Then   
$\ho{3}{\spl{2}{\mathcal{O}_v}}{\zhalf{\Z}}\cong \zhalf{\kind{\mathcal{O}_v}}\cong 
\zhalf{\kind{F}}$ by Proposition \ref{prop:h3afin} and there is a (split) short exact sequence 
\[
0\to \ho{3}{\spl{2}{\mathcal{O}_v}}{\zhalf{\Z}}\to
\ho{3}{\spl{2}{F}}{\zhalf{\Z}}\to \zhalf{\pb{k}}\to 0
\] 
(since $\zhalf{\rpbker{k}}=\zhalf{\pb{k}}$ for a finite field).
Recall that if $k=\F{q}$ then $\zhalf{\pb{k}}$ is cyclic of order $(q+1)_{\mathrm{odd}}$. 
\end{exa}
\begin{exa} In particular, if $p\geq 11$ is prime then 
$\ho{3}{\spl{2}{\Z_p}}{\zhalf{\Z}}\cong \zhalf{\kind{\Q_p}}$ and 
\[
\ho{3}{\spl{2}{\Q_p}}{\zhalf{\Z}}\cong \ho{3}{\spl{2}{\Z_p}}{\zhalf{\Z}}\oplus \Z/(p+1)_{\mathrm{odd}}.
\]
\end{exa}
\begin{exa} Let $k$ be any (sufficiently large) field. By Theorem \ref{thm:loc} there is a 
short exact sequence of $\sgr{k}$-modules
\[
\xymatrix{
0\ar[r]
& \ho{3}{\spl{2}{\pows{k}{x}}}{\zhalf{\Z}}\ar[r]
&\ho{3}{\spl{2}{\laurs{k}{x}}}{\zhalf{\Z}}\ar[r]^-{\delta_x}
&\zhalf{\rpbker{k}}\ar[r]
&0.\\
}
\]

Note further that the surjective maps 
\[
\ho{3}{\spl{2}{\pows{k}{x}}}{\zhalf{\Z}}\to \zhalf{\kind{\pows{k}{x}}}\mbox{ and }
\zhalf{\rpbker{k}}\to \zhalf{\pb{k}}
\]
both have kernels isomorphic to $\aug{k}\zhalf{\rpbker{k}}$.  This module is in general quite 
large and has been computed for many fields in \cite{hut:arxivh3sl2dv}. 
\end{exa}

\begin{exa} Let $p$ be any prime and take $k=\Q_p$ in the last example. Then 
$\aug{k}\zhalf{\rpbker{k}}\cong \zhalf{\pb{\F{p}}}\cong \Z/(p+1)_{\mathrm{odd}}$ (by 
\cite[Proposition 7.1]{hut:arxivh3sl2dv}). Thus there are  (split) short exact sequences
\[
0\to \zhalf{\pb{\F{p}}}\to \ho{3}{\spl{2}{\pows{\Q_p}{x}}}{\zhalf{\Z}}\to 
\zhalf{\kind{\pows{\Q_p}{x}}}\to 0
\] 
and
\[
0\to \zhalf{\pb{\F{p}}}\to \zhalf{\rpbker{\Q_p}}\to \zhalf{\pb{\Q_p}}\to 0.
\]
\end{exa}

\subsection{Acknowledgement} 

I wish to  thank  Behrooz Mirzaii for answering some questions 
concerning his results on the third homology of $SL_2$. 
%\section{Concluding discussion}

\pagebreak
\appendix

\section{A useful identity in $\rpb{A}$}\label{sec:app}
 Let $A$ be a local ring whose  residue field has at least 
$10$ elements. The purpose of this appendix is to prove the important identity 
\[
\pf{a}\cconst{A}=\suss{1}{a}-\suss{2}{a}
\]
 in $\rpb{A}$ for all $a\in \wn{A}$ (Theorem \ref{thm:df} below). 

This 
identity was proved in the 
case when $A$ is a field in \cite{hut:rbl11}. We verify here that the proof for fields adapts to 
the case of local rings with only minor modifications.   

Let $t$ denote the matrix of order $3$ 
\[
\matr{-1}{-1}{1}{0}\in \spl{2}{\Z}.
\]
It can be shown that 
$\ho{3}{\spl{2}{\Z}}{\Z}$ is cyclic of order $12$ and that the inclusion $G:=\an{t}\to \spl{2}{\Z}$ induces an 
isomorphism 
$\ho{3}{G}{\Z}\cong\ptor{\ho{3}{\spl{2}{\Z}}{{\Z}}}{3}$. 

We will identify $A$ with $A_+=\iota_+(A)\subset X_1$. 

For any subset $S$ of $A$, we have $S^n\subset X_1^n$ and we let 
\[
\discn{n}{S}:=S^n\cap X_n = \{ (s_1,\ldots,s_n)\in S^n\ |\ s_i-s_j\in \unitr{A}\mbox{ for } i\not= j\}.
\]
Thus there is a natural inclusion of additive groups $\Z[\discn{n}{S}]\to \Z[X_n]$.
 
We note that $\wn{A}$ is a right $G$-set since for any $a\in \wn{A}$
\[
a\cdot t = [a,1]\cdot \matr{-1}{-1}{1}{0}= [1-a,-a]=[1-a^{-1},1]=1-a^{-1}\in \wn{A}\subset X_1.
\]

For a local ring $A$, let $\twn{A}:=\wn{A}\setminus\{ a\in A|\ \Phi(\pi(a))=0 \mbox{ in } k\}$.

\begin{lem} \label{lem:twn}
%$\twn{A}$ is a right $G$-subset of $\wn{A}$; i.e. $\twn{A}\cdot t \subset \twn{A}$.
If $x\in \twn{A}$, then $x\cdot t^i-x\cdot t^j\in \unitr{A}$ whenever $i\not\equiv j\pmod{3}$.
\end{lem}
\begin{proof}
If $x\in \twn{A}$, then $x,1-x,\Phi(x)\in \unitr{A}$. The statement thus follows from the identities
\begin{eqnarray*}
x-x\cdot t &=& x-1+\frac{1}{x}=\frac{\Phi(x)}{x}\\
x-x\cdot t^2 &=& x-\frac{1}{1-x}=\frac{\Phi(x)}{x-1}\\
x\cdot t-x\cdot t^2 &=& 1-\frac{1}{x}-\frac{1}{1-x}=\frac{\Phi(x)}{x(x-1)}.\\
\end{eqnarray*}
\end{proof}

\begin{lem}\label{lem:beta} 
Let $A$ be a local ring whose residue field, $k$, has at least $10$ elements. 
Let $L_n=L_n(A)= \Z[X_{n+1}]$ as above. Let $F_\bullet$ be the (right) homogeneous standard resolution of $\Z$ over 
$\Z[G]$. Then an augmentation-preserving chain map of right $\Z[G]$-modules 
$F_\bullet\to L_\bullet$ in dimensions three and below  can be constructed as follows:

Let $x\in \twn{A}$ and let $y\in \twn{A}$ with 
\[
\pi(y)\not\in \{ \pi(x), \pi(x)\cdot t, \pi(x)\cdot t^2\}\subset  k.
\]

Then define $\beta_n^{x,y}:F_n\to L_n$ as follows:

Given  $g\in G$,  let $\beta_0^{x,y}(g)=x\cdot g \in A\subset X_1$.

Given $g_0,g_1\in G$ let 
\[
\beta_1^{x,y}(g_0,g_1)=
\left\{
\begin{array}{ll}
(x\cdot g_0,x\cdot g_1),& g_0\not= g_1\\
0,& g_0=g_1\\
\end{array}
\right.
\]

Given $g_0,g_1,g_2\in G$, let 
\[
\beta_2^{x,y}(g_0,g_1,g_2)=
\left\{
\begin{array}{ll}
(x\cdot g_0,x\cdot g_1,x\cdot g_2),& \mbox{ if } g_0,g_1,g_2 \mbox{ are distinct}\\
0,& g_0=g_1\mbox{ or } g_1=g_2\\
(y\cdot g_0,x\cdot g_0,x\cdot g_1)&\\
+(y\cdot g_0,x\cdot g_1,x\cdot g_0),& \mbox{ if }g_0=g_2\not= g_1\\
\end{array}
\right.
\]

Given $g_0,g_1,g_2,g_3\in G$ let
\[
\beta_3^{x,y}(g_0,g_1,g_2,g_3)=
\left\{
\begin{array}{ll}
0,& g_i=g_{i+1}\mbox{ for some } 0\leq i\leq 2 \\
(y\cdot g_0,y\cdot g_1,x\cdot g_0,x\cdot g_1)&\\
+(y\cdot g_0,y\cdot g_1,x\cdot g_1,x\cdot g_0),& \mbox{ if }g_0=g_2\not= g_1\mbox{ and } g_1=g_3\\
(x\cdot g_0,y\cdot g_0,x\cdot g_1,x\cdot g_2)&\mbox{ if } g_1=g_3 \mbox{and }\\
+(x\cdot g_0,y\cdot g_1,x\cdot g_2,x\cdot g_1),& g_0,g_1,g_2\mbox{ are distinct}\\
(y\cdot g_0,x\cdot g_0,x\cdot g_1,x\cdot g_3)&\mbox{ if } g_0=g_2 \mbox{ and }\\
-(y\cdot g_0,x\cdot g_1,x\cdot g_0,x\cdot g_3),& g_0,g_1,g_3\mbox{ are distinct}\\
(y\cdot g_0,x\cdot g_1,x\cdot g_2,x\cdot g_0)&\mbox{ if } g_0=g_3 \mbox{ and}\\
-(y\cdot g_0,x\cdot g_0,x\cdot g_1,x\cdot g_2),& g_0,g_1,g_2\mbox{ are distinct}\\
\end{array}
\right.
\]
\end{lem}
\begin{proof} By Lemma \ref{lem:twn}, the image of $\beta_n^{x,y}$ lies in $\Z[\discn{n}{A}]\subset \Z[X_{n+1}]$. It is a 
straightforward computation to verify that this map is an augmentation-preserving chain map. 
\end{proof}

\begin{cor}\label{cor:beta}
Let $x,y$ be chosen as in Lemma \ref{lem:beta}.  Let
\[
C:=C(x,y)=\rcr(y,x\cdot t,x\cdot t^2,x)-\rcr(y,x,x\cdot t,x\cdot t^2)+\rcr(y,y\cdot t,x,x\cdot t) +
\rcr(y,y\cdot t,x\cdot t,x)\in \rpb{A}.
\]
Then $C$ is independent of the choice of $x,y$, $C\in \rbl{A}$ and  $3C=0$. 
\end{cor}
\begin{proof}
Let $K\to \spl{2}{A}$ be any group homomorphism. Then there is a hypercohomolgy spectral sequence 
\[
E^1_{p,q}(K)=\ho{p}{K}{L_q}\Longrightarrow \ho{p+q}{K}{L_\bullet}
\]
and $\ho{p+q}{K}{L_\bullet}=\ho{p+q}{K}{\Z}$ when $p+q$ is not too large. 
There are associated edge homomorphisms giving a commutative diagram
\[
\xymatrix{
\ho{n}{K}{\Z}\ar[r]\ar[d]\ar^{\alpha_n(K)}[rd]
&E^2_{0,n}(K)=H_n((L_\bullet)_K)\ar[d]\\
\ho{n}{\spl{2}{A}}{\Z}\ar[r]
&
E^2_{0,n}(\spl{2}{A})=H_n((L_\bullet)_{\spl{2}{A}})
}
\]

The map $\alpha_n(K)$ can be constructed as follows: Let $F_\bullet(K)$ be a projective right $\Z[K]$-resolution of 
$\Z$ and let $\beta_\bullet:F_\bullet(K)\to L_\bullet$ be any augmentation-preserving homomorphism of right 
$\Z[k]$-complexes. Then $\alpha_n(K)$ is the induced map on $n$th homology groups associated to the map of 
complexes $F_\bullet(K)\otimes_{\Z[k]}\Z \to L_\bullet\otimes_{\Z[\spl{2}{A}]}\Z$. The map $\alpha_n(K)$ is independent of the 
choices of resolution $F_\bullet(K)$ and chain map $\beta_\bullet$.  

Applying this  to $G$, $F_\bullet$ and $\beta_3^{x,y}$ as in Lemma \ref{lem:beta}, and observing that the cycle 
$(1,t,t,t^2)+(1,t,t^2,1)+(1,t,1,t)\in F_3$ represents a generator of $\ho{3}{G}{\Z}\cong \Z/3$, we see that the map 
$\alpha_3:\Z/3=\ho{3}{G}{\Z}\to H_3((L_\bullet)_{\spl{2}{A}})$ sends $1$ to the class of 
\[
(y,x\cdot t,x\cdot t^2,x)-(y,x,x\cdot t,x\cdot t^2)+(y,y\cdot t,x,x\cdot t) +
(y,y\cdot t,x\cdot t,x)\in L_3
\] 

But the proof of Theorem \ref{thm:h3sl2A} shows that  $H_3((L_\bullet)_{\spl{2}{A}})\cong\rpbker{A}$ and that this 
isomorphism is induced by the refined cross ratio map.
 
Thus $1\in \Z/3$ maps to $C=C(x,y)\in \rpbker{A}$ under $\alpha_3$. It follows that $C$ is independent of $x,y$ and 
that $3C=0$.

Finally, the proof of Theorem \ref{thm:h3sl2A} shows that, at least after tensoring with $\zhalf{\Z}$,
 the image of the edge homomorphism lies in $E^\infty_{0,3}=\zhalf{\rbl{A}}$. Since $3C=0$, it follows that 
$C\in \rbl{A}$.
\end{proof}
\begin{thm} \label{thm:df}
Let $A$ be a local ring whose residue field has at least $10$ elements. Then 
\begin{enumerate}
\item For all $x\in \unitr{A}$, $\pf{x}\cconst{A}=\suss{1}{x}-\suss{2}{x}$.
\item  $\pf{x}\cconst{A}=0$ if $x\in \unitr{A}$ is of the form $\pm\Phi(a)u^2$ for some $a,u\in \unitr{A}$. 
\end{enumerate}
\end{thm}  
\begin{proof}
Choose $x,y\in \twn{A}$ as in Lemma \ref{lem:beta}. Then, the calculations in the proof of Theorem 3.12 of 
\cite{hut:rbl11} show that 
\[
C=C(x,y)= \an{-\Phi(x)r}(\suss{2}{r}-\suss{1}{r}-\cconst{A})
\]
where 
\[
r=r(x,y)=\frac{x-y}{x-1-xy}\in \wn{A}.
\]

By the Corollary, $C$ has order $3$ and is independent of $x$ and $y$.  Now, by choice of $x$ and $y$, $r$ can 
assume any value in $\wn{A}$. In particular, we can arrange for $r$ to have the form $-u^2$ for some unit $u$.
Since $4C=4\cconst{A}=0$ and $2\suss{i}{-u^2}=0$ for $i=1,2$, multiplying by $4$ gives
\[
C=-\an{\Phi(x)}\cconst{A}
\]  
for any $x\in \twn{A}$.  

Given $x\in \twn{A}$ we can find $x'\in \twn{A}$ such that $xx'-1,x+x'-1\in \unitr{A}$ and 
\[
x'':= \frac{xx'-1}{x+x'-1}\in \twn{A}.
\]
Since $\Phi(x)\Phi(x')= \Phi(x'')\cdot (x+x'-1)^2$, we have 
\[
C= -\an{\Phi(x'')}\cconst{A}=-\an{\Phi(x)}\an{\Phi(x')}\cconst{A}=\an{\Phi(x)}C
\]
for any $x\in \twn{A}$, and hence $C=-\cconst{A}$. 

It follows that $\an{\Phi(x)}\cconst{A}=\cconst{A}$ for any $x\in \twn{A}$; i.e $\pf{\Phi(x)}\cconst{A}=0$ 
for all $x\in \twn{A}$. Since $\an{-1}\cconst{A}=\cconst{A}$ also, it follows that 
\[
\pf{\pm\Phi(x)u^2}\cconst{A}=0
\]
for all $x\in \twn{A}$ and $u\in \unitr{A}$.

We now have that 
\[
-\cconst{A}=C= \an{-\Phi(x)r}(\suss{2}{r}-\suss{1}{r}-\cconst{A})
\]
for all $r\in \wn{A}$ and some $x\in\twn{A}$.
Thus
\[
-\an{\Phi(x)r}\cconst{A}=\suss{2}{r}-\suss{1}{r}-\cconst{A}\imp \pf{r}\cconst{A}=\suss{1}{r}-\suss{2}{r}
\]
since $\an{\Phi{x}}\cconst{A}=\cconst{A}$. It follows that 
\[
\pf{r}\cconst{A}=\suss{1}{r}-\suss{2}{r}
\] 
for all $r\in \wn{A}$.  

Finally, let $x\in U_{1,A}=\unitr{A}\setminus\wn{A}$. Fix $r\in \wn{A}$. Then 
$\suss{i}{x}=\suss{i}{rx}-\an{x}\suss{i}{r}$ for $i=1,2$. Furthermore, $\pf{x}=\pf{rx}-\an{x}\pf{r}$.  It follows that
\begin{eqnarray*}
\pf{x}\cconst{A}&=&\pf{rx}\cconst{A}-\an{x}\pf{r}\cconst{A}\\
&=& \suss{1}{rx}-\suss{2}{rx}-\an{x}\suss{1}{r}+\an{x}\suss{2}{r}\\
&=& \suss{1}{x}-\suss{2}{x}
\end{eqnarray*}
as required.
\end{proof}

\bibliography{H3}
\end{document}